\pdfoutput=1
\documentclass[letterpaper]{amsart}

\usepackage[utf8]{inputenc}
\usepackage{lmodern}
\usepackage[T1]{fontenc}
\usepackage{amssymb}
\usepackage{enumitem}
\usepackage{mathtools}
\usepackage{graphicx}

\usepackage{tikz-cd}
\usepackage[pdfusetitle]{hyperref}
\def\arxiv#1{\href{http://arxiv.org/abs/#1}{\texttt{arXiv:#1}}}
\def\doi#1{\href{http://doi.org/#1}{doi:#1}}
\usepackage{cleveref}

\tikzcdset{arrow style=Latin Modern}

\mathtoolsset{mathic}
\DeclareMathAlphabet\mathbfit{OML}{cmm}{b}{it}

\newlist{enumarabic}{enumerate}{1}
\setlist[enumarabic]{font=\normalfont,label=(\arabic*),leftmargin=0.3in}
\newlist{enumroman}{enumerate}{1}
\setlist[enumroman]{font=\normalfont,label=(\roman*),leftmargin=0.3in}

\numberwithin{equation}{section}
\allowdisplaybreaks[4]

\theoremstyle{plain}
\newtheorem{theorem}{Theorem}[section]
\newtheorem{proposition}[theorem]{Proposition}
\newtheorem{lemma}[theorem]{Lemma}
\newtheorem{corollary}[theorem]{Corollary}

\theoremstyle{definition}

\newtheorem{remark}[theorem]{Remark}

\theoremstyle{remark}
\newtheorem*{acknowledgements}{Acknowledgements}

\let\newterm\emph

\def\cf{\emph{cf.}}

\let\epsilon\varepsilon
\let\phi\varphi
\let\emptyset\varnothing

\def\smallotimes{\mathbin{\mkern-1mu\vcenter{\hbox{$\scriptstyle\otimes$}}\mkern-1mu}}

\def\Z{\mathbb Z}

\def\at#1{\left.#1\right|}

\DeclareMathOperator{\Hom}{Hom}
\DeclareMathOperator{\Tor}{Tor}
\def\Ainf{A_{\infty}}

\DeclareMathOperator{\BB}{\mathbf{B}}

\def\kk{\Bbbk}
\def\BBone{\mathbf{1}}
\def\tBBone{\overset{\text{$\leftrightarrow$}}{\BBone}}

\DeclareMathOperator{\length}{length}

\def\deg#1{|#1|}

\def\susp{\mathbf{s}}
\def\desusp{\susp^{-1}}

\def\aa{\mathbfit{a}}
\def\laa{\overset{\text{$\leftarrow$}}{\aa}}
\def\raa{\overset{\text{$\rightarrow$}}{\aa}}
\def\taa{\overset{\text{$\leftrightarrow$}}{\aa}}

\def\bb{\mathbfit{b}}

\def\rbb{\overset{\text{$\rightarrow$}}{\bb}}
\def\tbb{\overset{\text{$\leftrightarrow$}}{\bb}}

\def\cc{\mathbfit{c}}

\def\lDelta{\overset{\text{$\leftarrow$}}{\Delta}}
\def\rDelta{\overset{\text{$\rightarrow$}}{\Delta}}
\def\tDelta{\overset{\text{$\leftrightarrow$}}{\Delta}}

\def\EE{\mathbf{E}}
\def\ldd{\lambda_{t}}

\def\swee#1#2{#1^{(#2)}}

\def\AA{\tilde{A}}

\let\KS\varkappa
\def\eqKS{\stackrel{\KS}{=}}

\let\shuffle\nabla

\def\PhiGM{\Phi^{\mathrm{GM}}}
\def\Phihga{\Phi^{\mathrm{hga}}}
\def\Psihgc{\Psi^{\mathrm{hgc}}}

\DeclareMathOperator{\Cobar}{\mathbf{\Omega}}
\def\CobarEl#1{\langle#1\rangle}

\def\ff{f_{1}}

\def\jj{\mathbfit{j}}

\def\sweehgc#1#2{#1^{(\!(#2)\!)}}

\def\AW{AW}
\def\AWu#1{\AW_{\mkern -1mu #1}}

\def\restr#1#2#3{#1[#2,#3]}

\def\selfmapright{\mathbin{\rotatebox{90}{$\circlearrowleft$}}}

\def\gE{\tilde{E}}
\def\gp{\tilde{p}}

\begin{document}

\title[An \texorpdfstring{\(\Ainf\)}{A∞}-version of the Eilenberg--Moore theorem]{An \boldmath{\(\Ainf\)}-version of the Eilenberg--Moore theorem}
\author{Matthias Franz}
\thanks{The author was supported by an NSERC Discovery Grant.}
\address{Department of Mathematics, University of Western Ontario,
  London, Ont.\ N6A\;5B7, Canada}
\email{mfranz@uwo.ca}

\subjclass[2020]{Primary 55R20; secondary 16E45, 55T20}

\begin{abstract}
  We construct an \(\Ainf\)-structure on the two-sided bar construction
  involving homotopy Gerstenhaber algebras (hgas).
  It extends the non-associative product defined by Carlson and the author
  and generalizes the dga structure on the one-sided bar construction due to Kadeishvili--Saneblidze.
  As a consequence, the multiplicative cohomology isomorphism
  from the Eilenberg--Moore theorem is promoted to a quasi-iso\-mor\-phism of \(\Ainf\)-algebras.

  We also show that the resulting product on the differential torsion product involving cochain algebras
  agrees with the one defined by Eilenberg--Moore and Smith, for all triples of spaces.
  This is a consequence of the following result, which is of independent interest:
  The strongly homotopy commutative (shc) structure on cochains inductively constructed by Gugenheim--Munkholm
  agrees with the one previously defined by the author for all hgas.
\end{abstract}

\maketitle

\section{Introduction}

The Eilenberg--Moore theorem is an important result in the theory of fibre bundles and fibrations.
Let \(p\colon E\to B\) be a Serre fibration, \(g\colon X\to B\) a map and \(\gE=g^{*}E\) the pull-back of~\(E\) along~\(g\),
\begin{equation}
  \label{eq:intro:pull-back}
  \begin{tikzcd}
    \gE \arrow{d}[left]{\gp} \arrow{r}{\tilde{g}} & E \arrow{d}{p} \\
    X \arrow{r}{g} & B \mathrlap{.}
  \end{tikzcd}
\end{equation}
Under suitable assumptions on the spaces involved and on the coefficient ring~\(\kk\)
(to be recalled in \Cref{thm:em-ainf}), the Eilenberg--Moore theorem
relates the singular cohomology~\(H^{*}(\gE)\) to the cohomologies of~\(B\),~\(E\) and~\(X\),
all with coefficients in~\(\kk\).
This is done via a spectral sequence converging to~\(H^{*}(\gE)\) with second page
\begin{equation}
  \label{eq:intro:Tor-E2}
  E_{2} = \Tor_{H^{*}(B)}\bigl(H^{*}(X),H^{*}(E)\bigr).
\end{equation}
The spectral sequence is one of algebras and converges to~\(H^{*}(\gE)\) as an algebra.

In fact, out of the underlying normalized singular cochain complexes, one can construct a certain complex
whose cohomology
\begin{equation}
  \label{eq:intro:Tor}
  \Tor_{C^{*}(B)}\bigl(C^{*}(X),C^{*}(E)\bigr),
\end{equation}
 called a differential torsion product, is naturally isomorphic to~\(H^{*}(\gE)\).
This underlying complex involves a suitable resolution of, say, \(C^{*}(E)\) over~\(C^{*}(B)\),
in the same way as the torsion product~\eqref{eq:intro:Tor-E2} can be constructed out of a resolution
of~\(H^{*}(E)\) over~\(H^{*}(B)\). In both cases the module structures are given by the map~\(p^{*}\).

The fact that cohomology algebras are graded commutative implies that \eqref{eq:intro:Tor-E2} has
a canonical algebra structure.
If the two-sided bar construction~\(\BB\bigl(H^{*}(X),H^{*}(B),H^{*}(E)\bigr)\) is used to compute \eqref{eq:intro:Tor-E2},
then the product is the componentwise product of~\(H^{*}(X)\) and~\(H^{*}(E)\) together with the shuffle product
in~\(\BB H^{*}(B)\).

This breaks down on the cochain level because there the cup product is not graded commutative anymore.
In particular, while the two-sided bar construction
\begin{equation}
  \label{eq:intro:bar-twosided}
  \BB\bigl(C^{*}(X),C^{*}(B),C^{*}(E)\bigr),
\end{equation}
is a canonical choice for the complex underlying the differential torsion product~\eqref{eq:intro:Tor},
a multiplication based on the shuffle product would not be a chain map.
The classical way to address this is to define a product on the differential torsion product~\eqref{eq:intro:Tor}
via (the transpose of) the shuffle map, without a cochain-level multiplication on the bar construction~\eqref{eq:intro:bar-twosided}.
This is possible because the shuffle map is a morphism of differential graded coalgebras (dgcs)
and also a homotopy inverse to the Alexander--Whitney map
underpinning the dga structure of singular cochains.
Still, a product on~\eqref{eq:intro:bar-twosided} is desirable in some applications.

Assume that \(X\) is a point, so that the pull-back~\(\gE\) is the fibre~\(F\) of the fibration.
Since we use normalized cochains, we have \(C^{*}(X)=\kk\), and the two-sided bar construction becomes one-sided,
\begin{equation}
  \label{eq:intro:bar-right}
  \BB\bigl(\kk,C^{*}(B),C^{*}(E)\bigr).
\end{equation}
In this situation, Kadeishvili--Saneblidze have defined a product on~\eqref{eq:intro:bar-right}
that turns the bar construction into a differential graded algebra (dga) and the canonical map
\begin{equation}
  \BB\bigl(\kk,C^{*}(B),C^{*}(E)\bigr) \to C^{*}(F),
  \qquad
  [b_{1}|\dots|b_{k}]e \mapsto \begin{cases}
    \at{e}_{F} & \text{if \(k=0\),} \\
    0 & \text{otherwise}
  \end{cases}
\end{equation}
into a quasi-isomorphism of dgas. This is done by adding correction terms to the shuffle product.
These correction terms involve certain natural operations
\begin{equation}
  E_{k}\colon C^{*}(X) \otimes C^{*}(X)^{\otimes k} \to C^{*}(X)
\end{equation}
for the normalized cochains of any space~\(X\). The operation~\(E_{1}\) is essentially the cup-\(1\)-product;
the others are certain multilinear variants of it. Together they turn \(C^{*}(X)\)
into what is called a homotopy Gerstenhaber algebra (hga), which is a specific kind of \(E_{2}\)-algebra.
An hga structure on a dga~\(A\) defines an associative multiplication on the bar construction~\(\BB A=\BB(\kk,A,\kk)\),
and the Kadeishvili--Saneblidze product on~\eqref{eq:intro:bar-right} is an extension of it.
See \Cref{sec:hga} for precise definitions.

In order to describe the cup product in the cohomology of a general pull-back, Carlson and the author~\cite{CarlsonFranz}
have defined a product on the two-sided bar construction~\eqref{eq:intro:bar-twosided},
again based on the hga structure of singular cochains. While not turning \eqref{eq:intro:bar-twosided}
into a dga, this product is a chain map and the canonical map
\begin{align}
  \label{eq:intro:def-f1}
  \BB\bigl(C^{*}(X),C^{*}(B),C^{*}(E)\bigr) & \to C^{*}(\gE), \\
  \notag 
  x[b_{1}|\dots|b_{k}]e &\mapsto \begin{cases}
    \gp^{*}(x)\,\tilde{g}^{*}(e) & \text{if \(k=0\),} \\
    0 & \text{otherwise}
  \end{cases}
\end{align}
multiplicative up to some homotopy.
This is enough to conclude that the map
\begin{equation}
  \Tor_{C^{*}(B)}\bigl(C^{*}(X),C^{*}(E)\bigr) \to H^{*}(\gE)
\end{equation}
induced in cohomology is multiplicative.

That the product on~\eqref{eq:intro:bar-twosided} is not associative anymore
and the map~\eqref{eq:intro:def-f1} not multiplicative seems unavoidable.
It is therefore natural to consider the next best thing beyond dgas, which are \(\Ainf\)-algebras.
Recall that an \(\Ainf\)-algebra is a cochain complex~\(A\) together with maps
\begin{equation}
  m_{n}\colon A^{\otimes n}\to A
\end{equation}
of degree~\(2-n\) for~\(n\ge2\) such that the ``product''~\(m_{2}\) is a chain map, associative up to the homotopy~\(m_{3}\).
The higher maps are homotopies between the different ways to compose the previous maps; see \Cref{sec:ainf} for a precise definition.
Any dga is canonically an \(\Ainf\)-algebra by defining \(m_{2}\) to be the product and all other~\(m_{n}\) to be \(0\).

Similarly, a morphism of \(\Ainf\)-algebras~\(f\colon A\Rightarrow B\) is a family of maps
\begin{equation}
  f_{n}\colon A^{\otimes n}\to B
\end{equation}
of degree~\(1-n\) for~\(n\ge1\) such that \(f_{1}\colon A\to B\) is a chain map, multiplicative
up to the homotopy~\(f_{2}\). The higher maps are homotopies between the different ways
to combine the lower-order maps with the structure maps of the two \(\Ainf\)-algebras.
The morphism~\(f\) is a quasi-isomorphism if \(f_{1}\) is so.

In this paper we prove the following:

\begin{theorem}
  \label{thm:main-1}
  The Carlson--Franz product extends to an \(\Ainf\)-structure on the two-sided bar construction~\(\BB\bigl(C^{*}(X),C^{*}(B),C^{*}(E)\bigr)\),
  and the map~\eqref{eq:intro:def-f1} to a morphism of \(\Ainf\)-algebras
  \begin{equation}
    f\colon \BB\bigl(C^{*}(X),C^{*}(B),C^{*}(E)\bigr) \Rightarrow C^{*}(\gE)
  \end{equation}
  (where the target is a dga).
  Hence \(f\) is a quasi-isomorphism of \(\Ainf\)-algebras
  whenever the map~\eqref{eq:intro:def-f1} is a quasi-isomorphism of complexes.
\end{theorem}

See \Cref{thm:ainf-bar-twosided} and~\Cref{thm:augm-ainf}.
This confirms a conjecture expressed in~\cite[Rem.~A.26]{CarlsonFranz}.
\Cref{thm:main-1} holds in fact for general hgas, and the maps~\(m_{n}\)
and~\(f_{n}\) are given explicitly in terms of the hga operations. In particular, they are natural
with respect to maps of triples of hgas.

That the quasi-isomorphism~\eqref{eq:intro:def-f1} can be extended to one of \(\Ainf\)-algebras
follows from general results about the transfer of \(\Ainf\)-structures, at least for field coefficients.
See \Cref{rem:transfer} for details.
However, in such an approach the resulting structure would depend on certain choices;
one can therefore not expect it to be natural.
Moreover, it would be unclear whether an \(\Ainf\)-structure exists for hgas in general,
without a quasi-isomorphism like~\eqref{eq:intro:def-f1}.
Even in the context of the Eilenberg--Moore theorem,
expressing the \(\Ainf\)-structure maps in terms of hga operations is useful because
some spaces (notably the classifying space of a torus) are known to be homotopy Gerstenhaber formal
in the sense that there is a quasi-iso\-mor\-phism of dgas from cochains to cohomology
annihilating all hga operations, \cf~\cite[Thm.~9.6]{Franz:homog},~\cite[Thm.~1.3]{Franz:gersten}.

Whenever the conclusion of the Eilenberg--Moore theorem holds, both the product defined by Carlson--Franz
and the one originally considered by Eilenberg--Moore~\cite{EilenbergMoore:1966} and Smith~\cite{Smith:1967}
correspond to the cup product in~\(C^{*}(\gE)\) and therefore agree.
We give a direct proof of this fact that does not appeal to the Eilenberg--Moore theorem
and holds without any assumptions on the spaces involved.

\begin{theorem}
  \label{thm:main-2}
  Let \(B\),~\(X\) and~\(E\) be spaces or simplicial sets.
  The Carlson--Franz product on~\(\Tor_{C^{*}(B)}\bigl(C^{*}(X),C^{*}(E)\bigr)\)
  (in other words, the one induced by the \(\Ainf\)-structure from \Cref{thm:main-1})
  agrees with the one previously defined by Eilenberg--Moore and Smith.
\end{theorem}

See \Cref{thm:comparison-products}. This result has simultaneously been obtained by Carlson~\cite{Carlson:product}.
It is a consequence of a result about
strongly homotopy commutative (shc) algebra structures on singular cochains.

Recall that an shc structure on a dga~\(A\) is an extension of the multiplication map \(A\otimes A\to A\)
to an \(\Ainf\)-morphism~\(A\otimes A\Rightarrow A\).
Gugenheim--Munkholm~\cite{GugenheimMunkholm:1974} constructed a natural shc structure on~\(C^{*}(X)\)
based on the Eilenberg--Zilber maps. On the other hand, we have shown in~\cite{Franz:hgashc} that
any hga has a canonical shc structure, given explicitly in terms of the hga operations.
In this paper we prove the following, see \Cref{thm:Phihga-PhiGM}.

\begin{theorem}
  \label{thm:main-3}
  Let \(X\) be a space or a simplicial set.
  The shc structure on~\(C^{*}(X)\) constructed by Gugenheim--Munkholm coincides
  with the one coming from the hga structure of~\(C^{*}(X)\).
\end{theorem}

While not being part of the original construction,
the Gugenheim--Munkholm shc map can be used to define the product on~~\(\Tor_{C^{*}(B)}\bigl(C^{*}(X),C^{*}(E)\bigr)\).
\Cref{thm:main-2} follows easily from this fact together with \Cref{thm:main-3}.

Since the multiplication on the bar construction~\(\BB C^{*}(X)\) is determined by
the shc map~\(C^{*}(X)\otimes C^{*}(X)\Rightarrow C^{*}(X)\) associated to the hga structure,
we also recover a result of Hess--Parent--Scott--Tonks~\cite{HessEtAl:2006} about the relation
between the Gugenheim--Munkholm shc structure and the multiplication in~\(\BB C^{*}(X)\), see \Cref{thm:mult-bar-GM}.

We note that Carlson~\cite{Carlson:ring} studies ring structures on differential torsion products
for other \(E_{2}\)-algebras beyond hgas as well as commutative ring structures for \(E_{3}\)-algebras.
In the context of the Eilenberg--Moore theorem it is clear that the resulting product on the differential torsion product
is graded commutative since it computes \(H^{*}(\gE)\).
We have not investigated what happens for general hgas. Presumably,
one could get a graded commutative product by assuming certain additional structures,
for example the `extended hga' operations considered in~\cite[Sec.~3.2]{Franz:hgashc}.

The structure of the paper is as follows: After reviewing background material in \Cref{sec:prelim},
we prove \Cref{thm:main-1} in \Cref{sec:ainf-twosided} and spell out the resulting \(\Ainf\)-version of the
Eilenberg--Moore theorem in \Cref{sec:e-m}. In \Cref{sec:prelim-2} we review additional background material
that is needed for proving \Cref{thm:main-3} in \Cref{sec:compare-shc} and \Cref{thm:main-2} in \Cref{sec:compare-prod}.

\begin{acknowledgements}
  Discussions with Jeff Carlson, Dan Christensen and Anibal Medina-Mardones have led to
  \Cref{rem:fresse,rem:m2-hga} and to \Cref{foot:triv-ext}.
\end{acknowledgements}

\section{Preliminaries}
\label{sec:prelim}

\subsection{Generalities}

We work over a commutative ring~\(\kk\) with unit. All tensor products are taken over~\(\kk\).
All complexes are chain complexes (with a differential~\(d\) of degree~\(-1\))
or cochain complexes (with a differential of degree~\(+1\)) of \(\kk\)-modules.
The identity map on a complex~\(C\) is denoted by~\(1_{C}\) and the degree of an element~\(c\) by~\(\deg{c}\).
The differential on the desuspension~\(\desusp C\) of a complex~\(C\) is given by
\begin{equation}
  \label{eq:d-desusp}
  d\,\desusp c=-\desusp\,dc
\end{equation}
for~\(c\in C\).

Let \(A\) and~\(B\) be complexes, and let \(f\colon A\to B\) be a linear map of degree~\(\deg{f}\in\Z\).
The differential of~\(f\) is the linear map
\begin{equation}
  \label{eq:diff-map}
  d(f) = d_{B}\,f - (-1)^{\deg{f}}\,f\,d_{A}.
\end{equation}
of degree~\(\deg{f}\pm 1\).

Let \(A\) be an augmented dga with multiplication~\(\mu_{A}\colon A\otimes A\to A\) augmentation~\(\epsilon_{A}\colon A\to\kk\),
and let \(C\) be a coaugmented coalgebra with comultiplication~\(\Delta_{C}\colon C\to C\otimes C\) and coaugmentation~\(\iota_{C}\colon\kk\to C\).
A twisting cochain is a linear map~\(t\colon C\to A\) of degree~\(\pm 1\) such that
\begin{equation}
  d(t) = t\cup t\coloneqq \mu_{A}\,(t\otimes t)\,\Delta_{C},
  \qquad
  t\,\iota_{C} = 0,
  \qquad
  \epsilon_{A}\,t = 0.
\end{equation}

\subsection{Sign convention}

Most signs in our formulas come from the Koszul sign rule according to which the transposition of two graded objects~\(x\) and~\(y\)
incurs the sign~\((-1)^{\deg{x}\deg{y}}\). (We have seen this already in~\eqref{eq:d-desusp}.)
In order to avoid excessive sign exponents, we adopt the sign convention of~\cite{Franz:hgashc} and~\cite{Franz:homog}:
Whenever we write equalities with the symbol~``\(\eqKS\)'', we omit the signs that come from permuting variables
or from moving maps past variables (from the left). For example, if we write
\begin{equation}
  f(a,b,c) \eqKS a \otimes g(c, h(b)),
\end{equation}
then this stands for
\begin{equation}
  f(a,b,c) = (-1)^{\epsilon}\, a \otimes g(c, h(b))
  \quad\text{with}\quad
  \epsilon = \deg{b} \deg{c} + \deg{g} \deg{a} + \deg{h}(\deg{a}+\deg{c}).
\end{equation}

\goodbreak
\subsection{Bar constructions}

\subsubsection{The usual bar construction}

Let \(A\) be an augmented dga. We write \(\iota\colon\kk\to A\) for the unit map of~\(A\),
\(\epsilon\colon A\to\kk\) for the augmentation
and \(\bar a=a-\iota(\epsilon(a))\) for the component of~\(a\in A\) in the augmentation ideal~\(\bar A=\ker\epsilon\lhd A\).

The (reduced) bar construction of~\(A\) is the dgc
\begin{equation}
  \BB A = \bigoplus_{k\ge0} (\desusp \bar A)^{\otimes k}
\end{equation}
with diagonal
\begin{equation}
  \label{eq:Delta-BA}
  \Delta\,[a_{1}|\dots|a_{k}] = \sum_{i=0}^{k} [a_{1}|\dots|a_{i}] \otimes [a_{i+1}|\dots|a_{k}],
\end{equation}
augmentation
\begin{equation}
  \epsilon([a_{1}|\dots|a_{k}]) = \begin{cases}
    1 & \text{if \(k=0\),} \\
    0 & \text{otherwise,}
  \end{cases}
\end{equation}
coaugmentation
\begin{equation}
  \iota\colon \kk\to \BB A,
  \qquad
  1 \mapsto \BBone \coloneqq [\,]
\end{equation}
and differential
\begin{multline}
  \label{eq:d-BA}
  \qquad
  d\,[a_{1}|\dots|a_{k}] = - \sum_{i=0}^{k} (-1)^{\epsilon_{i}}\,[a_{1}|\dots|da_{i}|\dots|a_{k}] \\
  + \sum_{i=1}^{k-1} (-1)^{\epsilon_{i}}\,[a_{1}|\dots|a_{i}a_{i+1}|\dots|a_{k}]
  \qquad
\end{multline}
where
\begin{equation}
  \epsilon_{i} = \deg{a_{1}} + \dots + \deg{a_{i}} - i.
\end{equation}
(The first line in~\eqref{eq:d-BA} is the tensor product differential of~\((\desusp\bar A)^{\otimes k}\).)

We write elements of~\(\BB A\) also in the form
\begin{equation}
  \aa = [a_{1}|\dots|a_{k}] \,;
\end{equation}
we define the \newterm{length} of such an element to be \(k\).
We use the Sweedler notation without explicit sum signs. The formula~\eqref{eq:Delta-BA} for the diagonal then reads
\begin{equation}
  \Delta\,\aa = \swee{\aa}{1} \otimes \swee{\aa}{2}.
\end{equation}
We can consider the desuspension~\(\desusp a\) of some~\(a\in\bar A\) as an element~\([a]\in \BB A\) of length~\(1\).

There is the canonical twisting cochain
\begin{equation}
  t_{A}\colon \BB A\to A,
  \qquad
  [a_{1}|\dots|a_{k}] \mapsto \begin{cases}
    a_{1} & \text{if \(k=1\),} \\
    0 & \text{otherwise.}
  \end{cases}
\end{equation}

\subsubsection{One- and two-sided bar constructions}

Let \(A'\) and~\(A''\) also be augmented dgas. We assume that we have morphisms~\(A\to A'\) and~\(A\to A''\),
but we suppress these maps from our notation in order to make formulas more readable.

The two-sided bar construction associated to the triple~\((A',A,A'') \) is
\begin{equation}
  \BB(A',A,A'') = A' \otimes \BB A \otimes A'' \,;
\end{equation}
we write elements in the form
\begin{equation}
  \taa = a'\otimes\aa\otimes a'' = a'[a_{1}|\dots|a_{k}]a'' \,;
\end{equation}
the length of~\(\taa\) is the length of~\(\aa\).

Using the notation introduced above including the canonical twisting cochain~\(t=t_{A}\),
the differential on~\(\BB(A',A,A'')\) is
\begin{align}
  d\,\taa
  &= d_{\otimes}\,\taa + (-1)^{\deg{a}}\,a a_{1}[a_{2}|\dots|a_{k}]a'' - (-1)^{\deg{a}+\epsilon_{k-1}}\,a[a_{1}|\dots|a_{k-1}]a_{k} a'' \\
  \notag &\eqKS d_{\otimes}\,\taa
  + a'\,t(\swee{\aa}{1})\otimes\swee{\aa}{2}\otimes a'' - a'\otimes\swee{\aa}{1}\otimes t(\swee{\aa}{2})\,a''
\end{align}
where \(d_{\otimes}\) denotes the tensor product differential on~\(A' \otimes \BB A \otimes A''\). The second and third
term in the last line are only present for~\(k\ge1\). For later use we introduce the abbreviation
\begin{equation}
  \label{eq:def-ldd}
  \ldd\,\taa \eqKS a'\,t(\swee{\aa}{1})\otimes\swee{\aa}{2}\otimes a''
\end{equation}
for the ``left'' twisting term.

Setting \(A'=A''=\kk\) recovers the usual bar construction. Setting only one of them equal to~\(\kk\) gives
the left bar construction~\(\BB(A',A,\kk)\) and the right bar construction~\(\BB(\kk,A,A'')\), respectively.
We write elements in these bar constructions as
\begin{align}
  \laa &= a'\smallotimes\aa = a'[a_{1}|\dots|a_{k}] \in \BB(A',A,\kk), \\
  \raa &= \aa\smallotimes a'' = [a_{1}|\dots|a_{k}]a'' \in \BB(\kk,A,A'').
\end{align}

The diagonal~\eqref{eq:Delta-BA} of the dgc~\(\BB A\) generalizes to the chain map
\begin{gather}
  \label{eq:Delta-AAA}
  \tDelta\colon \BB(A',A,A'') \to \BB(A',A,\kk) \otimes \BB(\kk,A,A''), \\
  \notag \tDelta\,\taa 
  = \sum_{i=0}^{k} a'[a_{1}|\dots|a_{i}]\otimes[a_{i+1}|\dots|a_{k}]a'' =\vcentcolon \swee{\laa}{1}\otimes\swee{\raa}{2}.
\end{gather}
It contains the maps
\begin{align}
  \lDelta\colon \BB(A',A,\kk) &\to \BB(A',A,\kk) \otimes \BB A, \\
  \rDelta\colon \BB(\kk,A,A'') &\to \BB A \otimes \BB(\kk,A,A'')
\end{align}
as special cases. Together they have the ``coassociativity'' property
\begin{equation}
  (\lDelta\otimes 1)\,\tDelta = (1\otimes\rDelta)\,\tDelta\colon
  \BB(A',A,A'') \to \BB(A',A,\kk) \otimes \BB A \otimes \BB(\kk,A,A'').
\end{equation}

The shuffle product
\begin{equation}
  \label{eq:def-shuffle}
  \shuffle\colon \BB A \otimes \BB A \to \BB(A\otimes A),
  \qquad
  \aa\otimes\bb \mapsto \aa \bullet \bb
\end{equation}
(see for instance \cite[p.~247]{McCleary:2001} or~\cite[Ex.~2.1]{Franz:homog} for a definition)
is a morphism of dgcs and extends to the chain map
\begin{align}
  \label{eq:def-shuffle-twosided}
  \shuffle\colon \BB(A',A,A'') \otimes \BB(A',A,A'') &\to \BB(A'\otimes A',A\otimes A,A''\otimes A''), \\
  \notag \shuffle(\taa,\tbb) &\eqKS a'b' \otimes \aa \bullet \bb \otimes a''b''.
\end{align}
We call the latter one again a shuffle map, and we do the same for the corresponding maps in cohomology.

If \(A\) is graded commutative, then the composition of the shuffle map with the map~\(\BB\mu_{A}\)
induced by the multiplication map for~\(A\) turns \(\BB A\) into a dg~bialgebra.
Similarly, if also \(A'\) and~\(A''\) are cdgas, then one gets a dga structure on~\(\BB(A',A,A'')\).
This breaks down if one drops the commutativity assumptions. In \Cref{sec:hga,sec:ainf-twosided} we will see
how homotopy Gerstenhaber structures come to a rescue.

\subsubsection{Differential torsion products}
\label{sec:diff-tor}

The differential torsion product
\begin{equation}
  \label{eq:diff-Tor}
  \Tor_{A}(A',A'')
\end{equation}
is defined in~\cite[Def.~1.1]{GugenheimMay:1974} via so-called Künneth resolutions.
These resolutions include
the proper projective resolutions that are considered in~\cite[Def.~I.1.1]{Smith:1967},
and in~\cite[Def.~7.5]{McCleary:2001} under the assumption that \(A\) and~\(H^{*}(A)\) are flat over~\(\kk\).\footnote{%
  The assumption that \(H^{*}(A)\) be \(\kk\)-flat is missing in~\cite{Smith:1967}. It seems to be required to ensure
  that the cohomology of a proper projective \(A\)-module is again projective.}

Assume that \(\kk\) is a principal ideal domain and that \(A\) and~\(A''\) (or \(A\) and~\(A'\)) are torsion-free over~\(\kk\).
In this case, the cohomology of the two-sided bar construction~\(\BB(A',A,A'')\)
is naturally isomorphic to the differential torsion product~\eqref{eq:diff-Tor},
even if the bar construction itself may not be a Künneth resolution, see~\cite[p.~1148]{BarthelMayRiehl:2014}.
(Naturality follows from inspection of the proof of~\cite[Prop.~10.16]{BarthelMayRiehl:2014}.
Moreover, this isomorphism is compatible with the augmentation maps to~\(H^{*}(A'\otimes_{A}A'')\).

\subsection{Homotopy Gerstenhaber algebras}
\label{sec:hga}

\subsubsection{Definition}

A \newterm{homotopy Gerstenhaber algebra (hga)} is an augmented dga with certain additional operations
\begin{equation}
  E_{k}\colon A\otimes A^{\otimes k} \to A
\end{equation}
for~\(k\ge1\) that define a product
\begin{equation}
  \mu\colon \BB A \otimes \BB A \to \BB A,
  \qquad
  \mu(\aa,\bb) = \aa\cdot\bb,
\end{equation}
turning the bar construction~\(\BB A\) into a dg~bialgebra with unit~\(\BBone\), \cf~\cite[Def.~2, Sec.~3.2]{GerstenhaberVoronov:1995}.
The correspondence between the operations~\(E_{k}\) and the product~\(\mu\) is as follows:
As a morphism of dgcs, \(\mu\) is determined by the associated twisting cochain
\begin{equation}
  \EE = t_{A}\,\mu\colon \BB A\otimes \BB A\to \bar A.
\end{equation}
Being a twisting cochain means that
\begin{equation}
  \label{eq:d-EE}
  d(\EE)(\aa,\bb) \eqKS \EE(\swee{\aa}{1},\swee{\bb}{1})\,\EE(\swee{\aa}{2},\swee{\bb}{2})
\end{equation}
for any~\(\aa\),~\(\bb\in \BB A\).
That \(\BBone\in \BB A\) is a two-sided unit for~\(\mu\) and that \(\mu\) is associative
is encoded by the identities
\begin{gather}
  \label{eq:EE-unit}
  \EE(\aa,\BBone) = \EE(\BBone,\aa) = t_{A}(\aa) \\
  \label{eq:EE-assoc}
  \EE(\aa,\bb\cdot\cc) = \EE(\aa\cdot\bb,\cc)
\end{gather}
for~\(\aa\),~\(\bb\),~\(\cc\in \BB A\).
Among all such twisting cochains~\(\BB A\otimes \BB A\to A\),
the ones coming from an hga structure are exactly those satisfying
\begin{equation}
  \label{eq:EE-length-2}
  \EE(\aa,\bb) = 0 \quad\text{if \(\length\aa\ge2\)}.
\end{equation}
Note that \(\EE(\aa,\bb)\) is determined by~\eqref{eq:EE-unit} if \(\length\aa=0\) or~\(\length\bb=0\).
Finally, for~\(\aa=[a]\) and~\(\bb=[b_{1}|\dots|b_{k}]\) with~\(k\ge1\) we have
\begin{equation}
  \EE(\aa,\bb) = \pm E_{k}(a;b_{1},\dots,b_{k})
\end{equation}
where the sign is a matter of convention. One also defines \(E_{0}(a)=a\) for all~\(a\in A\). For one choice of signs,
the properties of the operations~\(E_{k}\) equivalent to~\eqref{eq:d-EE}--\eqref{eq:EE-assoc}
are spelled out in~\cite[Sec.~3.1]{Franz:hgashc} and~\cite[Sec.~6.1]{Franz:homog}.
In this paper, we will mostly work with the twisting cochain~\(\EE\).

Any graded commutative dga (cdga) is canonically an hga by setting \(E_{k}=0\) for all~\(k\ge1\).
This applies in particular to the cdga~\(\kk\). With this convention, the dga augmentation map~\(\epsilon\colon A\to\kk\)
of an hga becomes a morphism of hgas.

We point out that there is also a different convention where the roles of the first and the second argument of~\(\EE\)
are reversed, leading to operations
\begin{equation}
  \label{eq:tilde-E}
  \tilde{E}_{k}\colon A^{\otimes k}\otimes A\to A.
\end{equation}
See~\cite[Sec.~7]{KadeishviliSaneblidze:2005}, for instance.

Let \(X\) be a pointed topological space or simplicial set. Its normalized cochain complex~\(C^{*}(X)\)
with coefficients in~\(\kk\) has a natural hga structure,
see~\cite[Sec.~2.3]{GerstenhaberVoronov:1995}, \cite[Sec~8.2]{Franz:homog} or \Cref{sec:simplicial}.

\subsubsection{The product on the right bar construction}
\label{sec:prod-right-bar}

Kadeishvili--Saneblidze have defined an associative product on the right bar construction
of \(1\)-reduced hgas \cite[Cor.~6.2~\&~7.2]{KadeishviliSaneblidze:2005}.\footnote{%
  Since Kadeishvili--Saneblidze work with the operations we denoted \(\tilde E_{k}\) above,
  they obtain an associative product on the \emph{left} bar construction. Also, a wrong sign
  in their definition of the product is corrected in~\cite[Rem.~2]{Saneblidze:2009}.}
As remarked in~\cite[Prop.~7.6]{Franz:homog}, the definition works for all hgas.

Let \(A\) and~\(A''\) be hgas, connected via a morphism of hgas~\(A\to A''\) (which we again do not indicate).
The product in~\(\BB(\kk,A,A'')\) defined by Kadeishvili--Saneblidze has the form
\begin{equation}
  \label{eq:def-prod-bar-right}
  \raa\cdot\rbb \eqKS \sum \aa\cdot\swee{\bb}{1} \smallotimes \EE\bigl([\bar a''],\swee{\bb}{2}\bigr)\,b''
  + \aa\cdot\bb\smallotimes \epsilon(a'')\,b''.
\end{equation}
Recall that \([\bar a'']\) denotes the desuspension~\(\desusp\bar a''\) of~\(\bar a''=a''-\iota(\epsilon(a''))\in\bar A''\),
considered as an element of~\(\BB(A'')\).
The above product turns \(\BB(\kk,A,A'')\) into an augmented dga with unit \(\BBone_{\BB A}\smallotimes 1_{A''}\)
and augmentation~\(\epsilon_{\BB A}\smallotimes\epsilon_{A''}\).

\begin{lemma}
  \label{thm:diag-rb}
  For~\(\raa\),~\(\rbb\in \BB(\kk,A,A'')\) we have
  \begin{equation*}
    \rDelta(\raa\cdot\rbb) \eqKS \swee{\aa}{1}\cdot\swee{\bb}{1} \otimes \swee{\raa}{2}\cdot\swee{\rbb}{2} \in \BB A\otimes \BB(\kk,A,A'').
  \end{equation*}
\end{lemma}

\begin{proof}
  This follows from the definition~\eqref{eq:def-prod-bar-right} and the fact that \(\BB A\) is a bialgebra, so that
  \begin{equation}
    \Delta(\aa\cdot\bb) \eqKS \swee{\aa}{1}\cdot\swee{\bb}{1} \otimes \swee{\aa}{2}\cdot\swee{\bb}{2}
  \end{equation}
  for~\(\aa\),~\(\bb\in \BB A\).
\end{proof}

\subsection{\texorpdfstring{\(\Ainf\)}{A∞}-algebras and morphisms}
\label{sec:ainf}

We follow Keller~\cite[Sec.~3]{Keller:2001}, but we rewrite the formulas in a way
that uses the differential of a map as defined in~\eqref{eq:diff-map}.

An \(\Ainf\)-structure on a cochain complex~\(A\) is a family of map~\(m_{n}\colon A^{\otimes n}\to A\)
of degree~\(2-n\) for~\(n\ge2\) such that
\begin{multline}
  \label{eq:def-ainf-mul}
  d(m_{n})(a_{1},\dots,a_{n}) \eqKS \\*
  - \sum_{l=2}^{n-1}\sum_{i=0}^{n-l} (-1)^{i+l(n-i-l)}\,m_{n-l+1}\bigl(a_{1},\dots,m_{l}(a_{i+1},\dots,a_{i+l}),\dots,a_{n}\bigr)
\end{multline}
for all~\(n\ge2\) and all~\(a_{1}\),~\dots,~\(a_{n}\in A\).
It is strictly unital if there is a cocycle~\(1\in A\) such that \(m_{2}(a,1)=m_{2}(1,a)\) for all~\(a\in A\)
and additionally any~\(m_{n}\) with~\(n\ge3\) vanishes if some argument equals~\(1\).
Any dga is a strictly unital \(\Ainf\)-algebra by setting \(m_{2}\) equal to the product and \(m_{n}=0\) for~\(n\ge3\).

Regarding morphisms, we only need the special case where the target is a dga.
An \(\Ainf\)-map~\(f\colon A\Rightarrow B\) from an \(\Ainf\)-algebra~\(A\) to a dga~\(B\)
is a family of linear maps~\(f_{n}\colon A^{\otimes n}\to B\) of degree~\(1-n\), \(n\ge1\), such that
\begin{multline}
  \label{eq:def-ainf-map}
  d(f_{n})(a_{1},\dots,a_{n}) \eqKS
  \sum_{l=1}^{n-1} (-1)^{l}\,f_{l}(a_{1},\dots,a_{l})\cdot f_{n-l}(a_{l+1},\dots,a_{n}) \\
  + \sum_{l=2}^{n}\sum_{i=0}^{n-l} (-1)^{i+l(n-i-l)}\,f_{n-l+1}\bigl(a_{1},\dots,m_{l}(a_{i+1},\dots,a_{i+l}),\dots,a_{n}\bigr)
\end{multline}
for all~\(n\ge1\) and all~\(a_{1}\),~\dots,~\(a_{n}\in A\).
It is strict if \(f_{n}=0\) for~\(n\ge2\), in which case the above identities reduce to
\begin{align}
  d(f_{1}) &= 0, \\
  f_{1}(m_{2}(a_{1},a_{2})) &= f_{1}(a_{1})\cdot f_{1}(a_{2}), \\
  f_{1}(m_{n}(a_{1},\dots,a_{n})) &= 0
\end{align}
for~\(n\ge3\).
A morphism~\(f\colon A\Rightarrow B\) as before is strictly unital
if~\(f_{1}(1_{A})=1_{B}\) and any~\(f_{n}\) with~\(n>1\) vanishes if some argument equals \(1_{A}\).

Regarding augmentations, we limit ourselves to strict ones.
We say that a strictly unital \(\Ainf\)-algebra~\(A\) is augmented
if there is a strict morphism~\(\epsilon_{A}\colon A\to\kk\) such that \(\epsilon_{A}(1_{A})=1\).
An augmented morphism from an augmented \(\Ainf\)-algebra to an augmented dga~\(B\)
is a strictly unital \(\Ainf\)-morphism~\(f\colon A\Rightarrow B\) such that
\(\epsilon_{B}\,f_{1}=\epsilon_{A}\) and \(\epsilon_{B}\,f_{n}=0\) for~\(n\ge2\).
(This means that the \(\Ainf\)-composition~\(\epsilon_{B}\circ f\) equals \(\epsilon_{A}\), considered as an \(\Ainf\)-map.)

In the case of two augmented dgas~\(A\) and~\(B\),
augmented \(\Ainf\)-morphisms~\(f\colon A\Rightarrow B\)
are in bijection with morphisms of coaugmented dgcs~\(\BB f\colon \BB A\to\BB B\), \cf~\cite[Sec.~3.6]{Keller:2001}.
The composition of an augmented \(\Ainf\)-morphism~\(f\colon A\Rightarrow B\) with an augmented dga map \(g\colon B\to B'\)
is the augmented \(\Ainf\)-morphism~\(g\circ f\colon A\Rightarrow B'\) defined by
\begin{equation}
  (g\circ f)_{n} = g\,f_{n}\colon A^{\otimes n}\to B'.
\end{equation}
It corresponds to the morphism of dgcs~\(\BB g\,\BB f\colon \BB A\to \BB B'\).
Similarly, the composition of~\(f\) with an augmented dga map~\(h\colon A'\to A\)
is the augmented \(\Ainf\)-morphism \(f\circ h\colon A'\Rightarrow B\) defined by
\begin{equation}
  (f\circ h)_{n} = f_{n}\,h^{\otimes n}\colon (A')^{\otimes n}\to B,
\end{equation}
corresponding to the morphism of dgcs~\(\BB f\,\BB h\colon \BB A'\to \BB B.\)

\begin{remark}
  \(\Ainf\)-maps also induce maps between two-sided bar constructions. Assume that
  \begin{equation}
    \begin{tikzcd}
      A' \arrow[Rightarrow]{d}[left]{f'} & A \arrow{l} \arrow[Rightarrow]{d}[left]{f} \arrow{r} & A'' \arrow[Rightarrow]{d}[left]{f''} \\
      B' & B \arrow{l} \arrow{r} & B''
    \end{tikzcd}
  \end{equation}
  is a commutative diagram, where all objects are dgas,
  the horizontal maps dga morphisms and the vertical ones \(\Ainf\)-morphisms. (Everything is assumed to be augmented.)
  Such a diagram induces a chain map
  \begin{equation}
    \label{eq:ainf-twosidedbar}
    \BB(f',f,f'') \colon \BB(A',A,A'') \to \BB(B',B,B'')
  \end{equation}
  in a functorial way. Whenever the two-sided bar constructions compute the differential torsion products,
  we likewise get a map
  \begin{equation}
    \Tor_{f}(f',f'')\colon \Tor_{A}(A',A'') \to \Tor_{B}(B',B''),
  \end{equation}
  in cohomology, see \cite[Thm.~3.5\(_{*}\)]{GugenheimMunkholm:1974},~\cite[Prop.~1.26, Notation~1.33]{Carlson:biquot},
  or~\cite[Thm.~7]{Wolf:1977} for special cases.
\end{remark}

\section{The \texorpdfstring{\(\Ainf\)}{A∞}-structure on the two-sided bar construction}
\label{sec:ainf-twosided}

\subsection{Definition of maps}

Let \(A\),~\(A'\) and~\(A''\) be hgas with hga morphisms~\(A\to A'\) and~\(A\to A''\). As before, we do not indicate these maps.

For~\(A=A'\) there are a ``right shift'' map
\begin{equation}
  S\colon \BB(A,A,A'') \to \BB(\kk,A,A''),
  \quad
  a[a_{1}|\dots|a_{k}]a'' \mapsto [\bar a|a_{1}|\dots|a_{k}]a''
\end{equation}
of degree~\(-1\) and an ``augmentation map''
\begin{equation}
  \label{eq:def-f}
  \ff\colon \BB(A,A,A'') \to A'',
  \qquad
  \taa \mapsto a'\,\epsilon(\aa)\,a''.
\end{equation}
The latter is a chain map. Note that for~\(A''=\kk\) we get elements~\(S(\laa)\in \BB A\)
and \(\ff(\laa)=\epsilon(a')\,\epsilon(\aa)\in\kk\).

Our starting point for the \(\Ainf\)-structure on the two-sided bar construction
is the dga structure in the right bar construction~\(\BB(\kk,A,A'')\) recalled in~\Cref{sec:prod-right-bar}.
Based on it and on the map
\begin{gather}
  h\colon \BB(A,A,A'')\otimes \BB(A,A,A'') \to \BB(A,A,A''), \\
  \notag h(\taa,\tbb) \eqKS a'\otimes \raa\cdot S(\tbb),
\end{gather}
we define the family of maps
\begin{equation}
  h_{n}\colon \BB(A,A,A'')^{\otimes n} \to \BB(A,A,A'')
\end{equation}
of degree~\(1-n\) for~\(n\ge1\) by
\begin{align}
  h_{1}(\taa_{1}) &= \taa_{1}, \\
  h_{2}(\taa_{1},\taa_{2}) &= h(\taa_{1},\taa_{2}), \\
  h_{n}(\taa_{1},\dots,\taa_{n}) &\eqKS h_{n-1}\bigl(\taa_{1},\dots,\taa_{n-2},h(\taa_{n-1},\taa_{n})\bigr).
\end{align}

Combining \(h\) with the map~\(\ff\) introduced in~\eqref{eq:def-f},
we additionally define the maps
\begin{equation}
  \label{eq:def-fn}
  f_{n}\colon \BB(A,A,A'')^{\otimes n}\mapsto A'',
  \qquad
  f_{n}(\taa_{1},\dots,\taa_{n}) = \ff(h_{n}(\taa_{1},\dots,\taa_{n}))
\end{equation}
of degree~\(1-n\) for~\(n\ge2\).
Then 
\begin{equation}
  f_{n}(\taa_{1},\dots,\taa_{n})\eqKS f_{n-1}(\taa_{1},\dots,h(\taa_{n-1},\taa_{n}))
\end{equation}
for~\(n\ge2\). Note that we also have \(\ff(\taa_{1})=\ff(h_{1}(\taa_{1}))\) in analogy with~\eqref{eq:def-fn}.

Note that both~\(h_{n}\) and~\(f_{n}\) are morphisms of right \(A''\)-modules if \(A''\) acts from the right
on the last tensor factor of~\(\BB(A,A,A'')^{\otimes n}\).

We finally define maps
\begin{equation}
  m_{n}\colon \BB(A',A,A'')^{\otimes n} \to \BB(A',A,A'')
\end{equation}
of degree~\(2-n\) via
\begin{align}
  \label{eq:def-m2}
  m_{2}(\taa_{1},\taa_{2}) &\eqKS a'_{1}\,\EE\bigl(\swee{\aa_{1}}{1}, S\,h_{1}(\swee{\laa_{2}}{1})\bigr)
    \smallotimes \swee{\raa_{1}}{2}\cdot\swee{\raa_{2}}{2} \\
  \notag &\qquad + a'_{1}\,\epsilon(a'_{2})\smallotimes \raa_{1}\cdot\raa_{2} \\
  m_{n}(\taa_{1},\dots,\taa_{n}) &\eqKS a'_{1}\,\EE\bigl(\swee{\aa_{1}}{1}, S\,h_{n-1}(\swee{\laa_{2}}{1},\dots,\swee{\laa_{n}}{1})\bigr)
    \smallotimes \swee{\raa_{1}}{2}\cdots\swee{\raa_{n}}{2}
\end{align}
for~\(n\ge3\).
Here \(\swee{\aa_{1}}{1}\) is considered as an element of~\(\BB A'\)
and \(\swee{\laa_{2}}{1}\),~\dots,~\(\swee{\laa_{n}}{1}\) as elements of~\(\BB(A',A',\kk)\).
In the first term of~\eqref{eq:def-m2} we have written \(h_{1}(\swee{\laa_{2}}{1})\) instead of~\(\swee{\laa_{2}}{1}\)
to highlight the analogy with the formula for~\(n>2\).

\begin{remark}
  The map~\(m_{2}\) agrees with the non-associative product on~\(\BB(A',A,A'')\) defined by Carlson--Franz,
  see~\cite[Thm.~A.5]{CarlsonFranz} and also \Cref{rem:m2-hga}. Moreover, \(f_{2}\) is the homotopy~\(h\)
  from~\(f_{1}(\taa_{1})\,f_{1}(\taa_{2})\) to~\(f_{1}(m_{2}(\taa_{1},\taa_{2}))\)
  given in~\cite[eq.~(A.19)]{CarlsonFranz}.
\end{remark}

\subsection{Some lemmas}

\begin{lemma}
  \label{thm:S}
  For~\(\taa\in \BB(A,A,A'')\) we have
  \begin{align*}
    \rDelta\,S(\taa) &= \BBone \otimes S(\taa) + S(\swee{\laa}{1}) \otimes \swee{\raa}{2} \in \BB A\otimes \BB(\kk,A,A''), \\
    d(S)(\taa) &= \epsilon(a')\,\aa\smallotimes a'' - \BBone\smallotimes \ff(\taa) \in \BB(\kk,A,A'').
  \end{align*}
\end{lemma}

We omit the short proof.
The case~\(A''=\kk\) is can be found in~\cite[\S 2.3]{GugenheimMunkholm:1974}; there one gets in particular
\begin{equation}
  \label{eq:d-S-k}
  d(S)(\laa) = \epsilon(a')\,\aa - \epsilon(a')\,\epsilon(\aa)\,\BBone.
\end{equation}

\begin{lemma}
  \label{thm:Delta-h}
  For~\(n\ge1\) and~\(\taa_{1}\),~\dots,~\(\taa_{n}\in \BB(A,A,A'')\) we have
  \begin{multline*}
    \tDelta\,h_{n}(\taa_{1},\dots,\taa_{n}) \eqKS
      h_{n}(\swee{\laa_{1}}{1},\dots,\swee{\laa_{n}}{1}) \otimes \swee{\raa_{1}}{2}\cdots\swee{\raa_{n}}{2} \\
      + \sum_{k=1}^{n-1} h_{k}(\swee{\laa_{1}}{1},\dots,\swee{\laa_{k}}{1})
      \otimes \swee{\raa_{1}}{2}\cdots\swee{\raa_{k}}{2}\cdot S\,h_{n-k}(\taa_{k+1},\dots,\taa_{n}),
  \end{multline*}
  where the maps~\(h_{n}\) and~\(h_{k}\) on the right-hand side take values in the left bar construction~\(\BB(A,A,\kk)\).
\end{lemma}

\begin{proof}
  For~\(n=1\) our claim reduces to formula~\eqref{eq:Delta-AAA}. For~\(n=2\) it means
  \begin{equation}
    \tDelta\,h(\taa_{1},\taa_{2}) \eqKS
      h(\swee{\laa_{1}}{1},\swee{\laa_{2}}{1}) \otimes \swee{\raa_{1}}{2} \cdot \swee{\raa_{2}}{2}
      + \swee{\laa_{1}}{1} \otimes \swee{\raa_{1}}{2}\cdot S(\taa_{2}),
    \end{equation}
    which is a consequence of \Cref{thm:S,thm:diag-rb}.
    The claim for~\(n\ge3\) now follows by induction.
\end{proof}

\begin{lemma}
  For~\(n\ge2\) we have
  \begin{equation}
    \tDelta\,m_{n}(\taa_{1},\dots,\taa_{n}) \eqKS m_{n}(\swee{\laa_{1}}{1}\cdots\swee{\laa_{n}}{1}) \otimes \swee{\raa_{1}}{2}\cdots\swee{\raa_{n}}{2}
  \end{equation}
  where the map~\(m_{n}\) on the right-hand side takes values in the left bar construction~\(\BB(A',A,\kk)\).
\end{lemma}

\begin{proof}
  This holds by construction and \Cref{thm:diag-rb}.
\end{proof}

\begin{lemma}
  \label{thm:mn-h}
  For~\(n\ge2\) and~\(\taa_{1}\),~\dots,~\(\taa_{n+1}\in \BB(A,A,A'')\) we have
  \begin{multline*}
    \qquad
    m_{n}\bigl(\taa_{1},\dots,h(\taa_{n},\taa_{n+1})\bigr) \eqKS m_{n+1}(\taa_{1},\dots,\taa_{n+1}) \\
      + (-1)^{n}\,h\bigl(m_{n}(\taa_{1},\dots,\taa_{n}),\taa_{n+1}\bigr).
    \qquad
  \end{multline*}
\end{lemma}

\begin{proof}
  This follows from the definitions and the case~\(n=2\) of \Cref{thm:Delta-h}.
\end{proof}

\begin{lemma}
  \label{thm:d-h}
  For~\(n\ge1\) and~\(\taa_{1}\),~\dots,~\(\taa_{n}\in \BB(A,A,A'')\) we have
  \begin{multline*}
    d(h_{n})(\taa_{1},\dots,\taa_{n}) \eqKS
      \sum_{l=1}^{n-1} (-1)^{l}\, h_{l}(\taa_{1},\dots,\taa_{l})\cdot f_{n-l}(\taa_{l+1},\dots,\taa_{n}) \\
      + \sum_{l=2}^{n}\sum_{i=0}^{n-l} (-1)^{i+l(n-i-l)}\,h_{n-l+1}\bigl(\taa_{1},\dots,m_{l}(\taa_{i+1},\dots,\taa_{i+l}),\dots,\taa_{n}\bigr).
  \end{multline*}
\end{lemma}

\begin{proof}
  For~\(n=1\) we just assert that \(h_{1}\) is a chain map. For~\(n=2\) the claim is
  \begin{equation}
    d(h)(\taa_{1},\taa_{2}) = m_{2}(\taa_{1},\taa_{2}) - \taa_{1}\cdot \ff(\taa_{2}).
  \end{equation}
  To prove it, it is convenient to consider \(h\) as a map
  \begin{equation}
    A \otimes \BB(\kk,A,A'') \otimes \BB(A,A,A'') \to \BB(A,A,A''),
  \end{equation}
  where the domain is equipped with the tensor product differential~\(d_{\otimes}\).
  As in~\eqref{eq:def-ldd} we write the twisting term in the correct differential~\(d\) as~\(\ldd\).
  Using \Cref{thm:S}, one finds
  \begin{equation}
    (d\,h+h\,d_{\otimes})(\taa_{1},\taa_{2}) = h\bigl(\ldd(\taa_{1}),\taa_{2}\bigr) + m_{2}(\taa_{1},\taa_{2}) - \taa_{1}\cdot \ff(\taa_{2}),
  \end{equation}
  which gives the desired formula.
  The claim for~\(n>2\) follows by induction, \Cref{thm:mn-h} and the equivariance of~\(h_{n}\) with respect to the right \(A''\)-action.
\end{proof}

\begin{remark}
  \label{rem:d-h-k}
  Assume \(A''=\kk\). Then all bar elements in~\(h_{k}(\laa_{1},\dots,\laa_{k})\) have length at least~\(k-1\), so that
  \begin{equation}
    f_{k}(\laa_{1},\dots,\laa_{k}) = \ff(h_{k}(\laa_{1},\dots,\laa_{k}))
  \end{equation}
  vanishes for~\(k>1\). Hence the first sum on the right-hand side of \Cref{thm:d-h}
  reduces in this case to the term for~\(l=n-1\), which is
  \begin{equation}
    (-1)^{n-1}\,h_{n-1}(\laa_{1},\dots,\laa_{n-1})\,\epsilon(a'_{n})\,\epsilon(\aa_{n}).
  \end{equation}
\end{remark}

\subsection{Proof of the \texorpdfstring{\boldmath{\(\Ainf\)}}{A∞}-property}

We are now ready to prove our first main result.

\begin{theorem}
  \label{thm:ainf-bar-twosided}
  The family~\((m_{n})\) defines an augmented \(\Ainf\)-structure on~\(\BB(A',A,A'')\)
  with unit~\(1_{A'}\otimes\BBone_{\BB A}\otimes 1_{A''}\) and augmentation~\(\epsilon(\taa)=\epsilon(a')\,\epsilon(\aa)\,\epsilon(a'')\).
\end{theorem}

Note that for~\(A'=\kk\) this \(\Ainf\)-structure reduces to the Kadeishvili--Saneblidze dga structure
on~\(\BB(\kk,A,A'')\) since each hga operation~\(E_{k}\) vanishes on~\(\kk\) for~\(k\ge1\).

\begin{proof}
  We first verify the \(\Ainf\)-condition~\eqref{eq:def-ainf-mul}.
  For the moment, let us assume the case~\(n=2\) to hold and consider \(n\ge3\).
  For the purpose of this proof, let us also consider \(m_{n}\) as a map
  \begin{equation}
    m_{n}\colon A' \otimes \BB(\kk,A,A'') \otimes \BB(A',A,A'')^{\otimes(n-1)} \to A' \otimes \BB(\kk,A,A'')
  \end{equation}
  where both the domain and the target have the tensor product differential~\(d_{\otimes}\).
  As before, we denote both missing twisting terms by~\(\ldd\).

  We then have
  \begin{subequations}
  \begin{align}
    \label{eq:d-mn} \MoveEqLeft{d_{\otimes}(m_{n})(\taa_{1},\dots,\taa_{n})} \\
    \label{eq:eq-d-E} \eqKS {} & a'_{1}\,d(\EE)\bigl(\swee{\aa_{1}}{1}, S\,h_{n-1}(\swee{\laa_{2}}{1},\dots,\swee{\laa_{n}}{1})\bigr)\smallotimes \swee{\raa_{1}}{2}\cdots\swee{\raa_{n}}{2} \\
    \label{eq:eq-d-S} &- a'_{1}\,\EE\bigl(\swee{\aa_{1}}{1},d(S)\,h_{n-1}(\swee{\laa_{2}}{1},\dots,\swee{\laa_{n}}{1})\bigr)\smallotimes\swee{\raa_{1}}{2}\cdots\swee{\raa_{n}}{2} \\
    \label{eq:eq-d-h} &+ a'_{1}\,\EE\bigl(\swee{\aa_{1}}{1},S\,d(h_{n-1})(\swee{\laa_{2}}{1},\dots,\swee{\laa_{n}}{1})\bigr)\smallotimes\swee{\raa_{1}}{2}\cdots\swee{\raa_{n}}{2}.
  \end{align}
  \end{subequations}
  Let us consider each of the three terms on the right-hand side separately.

  We start with~\eqref{eq:eq-d-E}.
  \Cref{thm:S,thm:Delta-h} imply the identity
  \begin{multline}
    \Delta\,S\,h_{n-1}(\laa_{2},\dots,\laa_{n}) \eqKS \BBone \otimes S\,h_{n-1}(\laa_{2},\dots,\laa_{n}) \\
      + \sum_{l=1}^{n-2} S\,h_{l}(\swee{\laa_{2}}{1},\dots,\swee{\laa_{l+1}}{1})
        \otimes \swee{\aa_{2}}{2}\cdots\swee{\aa_{l+1}}{2}\cdot S\,h_{n-1-l}(\laa_{l+2},\dots,\laa_{n}) \\
      + S\,h_{n-1}(\swee{\laa_{2}}{1},\dots,\swee{\laa_{n}}{1}) \otimes \swee{\aa_{2}}{2}\cdots\swee{\aa_{n}}{2}.
  \end{multline}
  Combining this with the twisting cochain condition~\eqref{eq:d-EE}, we can write \eqref{eq:eq-d-E} as
  \begin{subequations}
  \begin{align}
    \notag &\eqKS a'_{1}\,\EE(\swee{\aa_{1}}{1},\BBone)\,\EE\bigl(\swee{\aa_{1}}{2},S\,h_{n-1}(\swee{\laa_{2}}{1},\dots,\swee{\laa_{n}}{1})\bigr)\smallotimes\swee{\raa_{1}}{3}\swee{\raa_{2}}{2}\cdots\swee{\raa_{n}}{2} \\
    \notag &\qquad + \sum_{l=1}^{n-2} (-1)^{l}\,\EE\bigl(\swee{\aa_{1}}{1},S\,h_{l}(\swee{\laa_{2}}{1},\dots,\swee{\laa_{l+1}}{1})\bigr) \\
    \notag &\qquad\qquad\qquad \cdot \EE\bigl(\swee{\aa_{1}}{2},\swee{\aa_{2}}{2}\cdots\swee{\aa_{l+1}}{2}\cdot S\,h_{n-1-l}(\swee{\laa_{l+2}}{1},\dots,\swee{\taa_{n}}{1})\bigr) \\
    \notag &\qquad\qquad\qquad\qquad\qquad\qquad \smallotimes\swee{\raa_{1}}{3}\cdots\swee{\raa_{l+1}}{3}\,\swee{\raa_{l+2}}{2}\cdots\swee{\raa_{n}}{2} \\
    \notag &\qquad + (-1)^{1-n}\,a'_{1}\,\EE\bigl(\swee{\aa_{1}}{1},S\,h_{n-1}(\swee{\laa_{2}}{1},\dots,\swee{\laa_{n}}{1})\bigr)\,\EE(\swee{\aa_{1}}{2},\swee{\aa_{2}}{2}\cdots\swee{\aa_{n}}{2}) \\
    \notag &\qquad\qquad\qquad\qquad\qquad\qquad \smallotimes\swee{\raa_{1}}{3}\cdots\swee{\raa_{n}}{3} \\
    \label{eq:a-1} &\eqKS a'_{1}\,t(\swee{\aa_{1}}{1})\,\EE\bigl(\swee{\aa_{1}}{2},S\,h_{n-1}(\swee{\laa_{2}}{1},\dots,\swee{\laa_{n}}{1})\bigr)\smallotimes\swee{\raa_{1}}{3}\swee{\raa_{2}}{2}\cdots\swee{\raa_{n}}{2} \\
    \label{eq:a-2} &\qquad + \sum_{l=2}^{n-1} (-1)^{l-1}\,\EE\bigl(\swee{\aa_{1}}{1},S\,h_{l-1}(\swee{\laa_{2}}{1},\dots,\swee{\laa_{l+1}}{1})\bigr) \\*
    \notag &\qquad\qquad\qquad \cdot \EE\bigl(\swee{\aa_{1}}{2}\cdots\swee{\aa_{l}}{2}, S\,h_{n-l}(\swee{\laa_{l+1}}{1},\dots,\swee{\taa_{n}}{1})\bigr) \\*
    \notag &\qquad\qquad\qquad\qquad\qquad\qquad \smallotimes\swee{\raa_{1}}{3}\cdots\swee{\raa_{l}}{3}\,\swee{\raa_{l+1}}{2}\cdots\swee{\raa_{n}}{2} \\
    \label{eq:a-3} &\qquad + (-1)^{1-n}\,a'_{1}\,\EE\bigl(\swee{\aa_{1}}{1},S\,h_{n-1}(\swee{\laa_{2}}{1},\dots,\swee{\laa_{n}}{1})\bigr)\,\EE(\swee{\aa_{1}}{2},\swee{\aa_{2}}{2}\cdots\swee{\aa_{n}}{2}) \\
    \notag &\qquad\qquad\qquad\qquad\qquad\qquad \smallotimes\swee{\raa_{1}}{3}\cdots\swee{\raa_{n}}{3}
  \end{align}
  \end{subequations}

  As in the special case~\eqref{eq:d-S-k} of \Cref{thm:S} we get for~\eqref{eq:eq-d-S} the term
  \begin{subequations}
  \begin{align}
    \label{eq:b-0} 
    &\eqKS -a'_{1}\,\EE\bigl(\swee{\aa_{1}}{1},\epsilon(a'_{2})\,\swee{\aa_{2}}{1}\,S\,h_{n-2}(\swee{\laa_{3}}{1},\dots,\swee{\laa_{n}}{1})\bigr) \smallotimes \swee{\raa_{1}}{2}\cdots\swee{\raa_{n}}{2} \\
    \label{eq:b-1} &\eqKS -a'_{1}\,\epsilon(a'_{2})\,\EE\bigl(\swee{\aa_{1}}{1}\cdot\swee{\aa_{2}}{1},S\,h_{n-2}(\swee{\laa_{3}}{1},\dots,\swee{\laa_{n}}{1})\bigr) \smallotimes \swee{\raa_{1}}{2}\cdots\swee{\raa_{n}}{2}.
  \end{align}
  \end{subequations}

  By \Cref{thm:d-h} and \Cref{rem:d-h-k} we have for~\eqref{eq:eq-d-h}
  \begin{subequations}
  \begin{align}
    \notag &\eqKS (-1)^{n-1}\,\EE\bigl(\swee{\aa_{1}}{1},S\,h_{n-2}(\swee{\laa_{2}}{1},\dots,\swee{\laa_{n-1}}{1})\,\epsilon(a'_{n})\,\epsilon(\swee{\aa_{n}}{1})\bigr) \smallotimes \swee{\raa_{1}}{2}\cdots\swee{\raa_{n}}{2} \\
    \notag &\qquad + \sum_{l=2}^{n-1}\sum_{i=0}^{n-1-l} (-1)^{i+l(n-1-i-l)}\,a'_{1} \\
    \notag &\qquad\qquad \cdot\EE\Bigl(\swee{\aa_{1}}{1}, S\,h_{n-l}\bigl(\swee{\laa_{2}}{1},\dots,m_{l}(\swee{\laa_{i+2}}{1},\dots,\swee{\laa_{i+1+l}}{1}),\dots,\swee{\laa_{n}}{1}\bigr)\Bigr) \\
    \notag &\qquad\qquad\qquad\qquad\qquad\qquad \smallotimes \swee{\raa_{1}}{2}\cdots\swee{\raa_{n}}{2} \\
    \label{eq:c-1} &\eqKS (-1)^{n-1}\,\EE\bigl(\swee{\aa_{1}}{1},S\,h_{n-2}(\swee{\laa_{2}}{1},\dots,\swee{\laa_{n-1}}{1})\bigr)\,\epsilon(a'_{n}) \smallotimes \swee{\raa_{1}}{2}\cdots\swee{\raa_{n-1}}{2}\cdot\raa_{n} \\
    \label{eq:c-2} &\qquad + \sum_{l=2}^{n-1}\sum_{i=0}^{n-1-l} (-1)^{i+l(n-1-i-l)}\,a'_{1} \\
    \notag &\qquad\qquad \cdot\EE\Bigl(\swee{\aa_{1}}{1}, S\,h_{n-l}\bigl(\swee{\laa_{2}}{1},\dots,m_{l}(\swee{\laa_{i+2}}{1},\dots,\swee{\laa_{i+1+l}}{1}),\dots,\swee{\laa_{n}}{1}\bigr)\Bigr) \\
    \notag &\qquad\qquad\qquad\qquad\qquad\qquad \smallotimes \swee{\raa_{1}}{2}\cdots\swee{\raa_{n}}{2}.
  \end{align}
  \end{subequations}

  Putting everything together, we obtain
  \begin{subequations}
  \begin{align}
    \notag \MoveEqLeft{d_{\otimes}(m_{n})(\taa_{1},\dots,\taa_{n})} \\
    \label{eq:res-1} \eqKS {} & (-1)^{2-n}\,m_{n}\bigl(\ldd(\taa_{1}),\dots,\taa_{n}\bigr) \\
    \label{eq:res-3} &- \ldd\bigl(m_{n}(\taa_{1},\dots,\taa_{n})\bigr) \\
    \label{eq:res-4} &- \sum_{l=2}^{n-1}\sum_{i=1}^{n-l} (-1)^{i+l(n-i-l)} \\*
    \notag &\qquad\qquad \cdot m_{n-l+1}\bigl(\taa_{1},\dots,m_{l}(\taa_{i+1},\dots,\taa_{i+l}),\dots,\taa_{n}\bigr) \\
    \label{eq:res-2} &- \sum_{l=2}^{n-1} (-1)^{l(n-l)}\,m_{n-l+1}\bigl(m_{l}(\taa_{1},\dots,\taa_{l}),\dots,\taa_{n}\bigr).
  \end{align}
  \end{subequations}
  Here we have rewritten \eqref{eq:a-1} as~\eqref{eq:res-1},
  \eqref{eq:a-3} as~\eqref{eq:res-3} and \eqref{eq:c-2} as~\eqref{eq:res-4}.
  \eqref{eq:res-2}, which is the missing case~\(i=0\) of~\eqref{eq:res-4}, essentially corresponds to~\eqref{eq:a-2}. However,
  for~\(l=2\) we have subsumed the term~\eqref{eq:b-1}, in line with the separate definition~\eqref{eq:def-m2} of~\(m_{2}\).
  Similarly, for~\(l=n-1\) we have incorporated the term~\eqref{eq:c-1} since \(n-l+1=2\) in this case.
  (If \(l=n-1=2\), then both terms appear.)

  To complete the proof for~\(n\ge3\),
  we replace the tensor product differentials by the twisted differentials including the twisting terms~\(\ldd\),
  \begin{equation}
    d(m_{n}) = d_{\otimes}(m_{n}) +\ldd\,m_{n} - (-1)^{2-n}\,m_{n}\,\bigl(\ldd\otimes 1^{\otimes(n-1)}\bigr),
  \end{equation}
  to get the desired result.

  For~\(n=2\) we have to show that \(m_{2}\) is a chain map. This is already contained in~\cite{CarlsonFranz},
  but we can also do it quickly now:
  We can still express \(d_{\otimes}(m_{2})\) as before 
  since the additional term \(a'_{1}\,\epsilon(a'_{2})\smallotimes \raa_{1}\cdot\raa_{2}\) in~\eqref{eq:def-m2} is a chain map
  with respect to the tensor product differentials.
  The differential in~\eqref{eq:eq-d-h} vanishes for~\(n=2\),
  the sums~\eqref{eq:a-2} and~\eqref{eq:c-2} are empty,
  and the term~\eqref{eq:b-0} becomes part of~\(\ldd(m_{2}(\taa_{1},\taa_{2}))\). Hence
  \begin{equation}
    d_{\otimes}(m_{2})(\taa_{1},\taa_{2}) = m_{2}\bigl(\ldd(\taa_{1}\taa_{2})\bigr) - \ldd\bigl(m_{2}(\taa_{1},\taa_{2})\bigr),
  \end{equation}
  which shows that \(m_{2}\) is a chain map with respect to the twisted differentials.

  We now turn to the claimed unit~\(\tBBone=1\otimes\BBone\otimes1\in \BB(A',A,A'')\).
  It is a cocycle and a two-sided identity for~\(m_{2}\). Now assume \(n\ge3\).
  Since \(S\,h_{n-1}(\laa_{2},\dots,\laa_{n})\) has length at least~\(1\), it follows that \(m_{n}(\taa_{1},\dots,\taa_{n})\)
  vanishes if \(\taa_{1}=\tBBone\). By induction one can show that \(S\,h_{n-1}(\laa_{2},\dots,\laa_{n})\) vanishes
  if one of the arguments is the identity, hence so does \(m_{n}\). This proves that the \(\Ainf\)-structure
  is strictly unital.

  The map~\(\taa\mapsto\epsilon(\taa)=\epsilon(a')\,\epsilon(\raa)\) is a chain map and multiplicative with respect to~\(m_{2}\).
  For~\(n\ge3\), the image of~\(m_{n}\) lies in \(\bar A'\otimes \BB(\kk,A,A'')\)
  because elements in the image of~\(S\,h_{n-1}\) have length at least~\(2\).
  Hence \(\epsilon\) vanishes on the image of~\(m_{n}\) for~\(n\ge3\).
  It therefore is a strict morphism~\(\BB(A',A,A'')\to\kk\),
  and it clearly send \(\tBBone\) to~\(1\in\kk\). So \(\epsilon\) is a augmentation for~\(\BB(A',A,A'')\), which completes the proof.
\end{proof}

\begin{remark}
  \label{rem:product-E2}
  Assume that \(\kk\) is a principal ideal domain
  and that \(A\) and~\(A'\) (or \(A\) and~\(A''\)) as well as their cohomologies are torsion-free over~\(\kk\).
  The base component~\(m_{2}\) of the \(\Ainf\)-structure on~\(\BB(A',A,A'')\) is compatible with the filtration
  by the length of elements. Hence we obtain a spectral sequence of algebras whose first page is
  \begin{equation}
    E_{1} = \BB\bigl(H^{*}(A'),H^{*}(A),H^{*}(A'')\bigr)
  \end{equation}
  by the Künneth theorem and our assumption of torsion-freeness.
  The product on this page is the composition of the shuffle product~\eqref{eq:def-shuffle-twosided}
  with the map
  \begin{equation}
    \BB\bigl(H^{*}(A')^{\otimes2},H^{*}(A)^{\otimes2},H^{*}(A'')^{\otimes2}\bigr)
    \to \BB\bigl(H^{*}(A'),H^{*}(A),H^{*}(A'')\bigr)
  \end{equation}
  induced by the multiplication maps for~\(H^{*}(A)\),~\(H^{*}(A')\) and~\(H^{*}(A'')\),
  which are morphisms of algebras because of graded commutativity.
  The reason that we get this product on~\(E_{1}\) is that all deformation terms
  involving the hga operations~\(E_{k}\) with~\(k\ge1\) lower the filtration degree.

  As a consequence, the second page of the spectral sequence is
  \begin{equation}
    E_{2} = \Tor_{H^{*}(A)}\bigl(H^{*}(A'),H^{*}(A'')\bigr)
  \end{equation}
  with the usual product.
\end{remark}

\subsection{The augmentation map}
\label{sec:augm}

Assume that in addition to the triple~\((A',A,A'')\) considered so far
we have another hga~\(\AA\) together with morphisms~\(A'\to\AA\) and~\(A''\to\AA\)
such that the diagram
\begin{equation}
  \begin{tikzcd}
    A \arrow{d} \arrow{r} & A' \arrow{d} \\
    A'' \arrow{r} & \AA
  \end{tikzcd}
\end{equation}
commutes. We would like to have an \(\Ainf\)-morphism~\(\BB(A',A,A'')\Rightarrow\AA\)
that extends the canonical chain map
\begin{equation}
  \label{eq:BAAA-A}
  \BB(A',A,A'')\to\AA
  \qquad
  a'\otimes\aa\otimes a'' \mapsto a'\,\epsilon(\aa)\,a'',
\end{equation}
where on the right-hand side both~\(a'\) and~\(a''\) are considered as elements of~\(\AA\);
compare the definition of~\(\ff\) in~\eqref{eq:def-f}.
From the naturality of the \(\Ainf\)-structure on the two-sided bar construction
we immediately get a strict \(\Ainf\)-morphism
\begin{equation}
  \BB(A',A,A'')\Rightarrow \BB(\AA,\AA,\AA).
\end{equation}
Se we can assume~\(A=A'=A''=\AA\) for our construction. In a slightly more general setting, we have the following.

\begin{proposition}
  \label{thm:augm-ainf}
  The family~\((f_{n})\) defined in~\eqref{eq:def-fn} constitutes an augmented \(\Ainf\)-morphism~\(f\colon \BB(A,A,A'')\Rightarrow A''\).
  Hence the composition
  \begin{equation*}
    \BB(A',A,A'')\rightarrow \BB(\AA,\AA,\AA) \xRightarrow{f} \AA
  \end{equation*}
  is an augmented \(\Ainf\)-morphism extending the map~\(\BB(A',A,A'')\to\AA\) from~\eqref{eq:BAAA-A}.
\end{proposition}

By abuse of notation, we write this composition of maps again as~\(f\).

\begin{proof}
  Applying the chain map~\(\ff\) to both sides of the identity from \Cref{thm:d-h} confirms that we have
  an \(\Ainf\)-morphism~\(f\colon \BB(A,A,A'')\Rightarrow A''\). It satisfies
  \begin{equation}
    \epsilon(f_{1}(\taa)) = \epsilon\bigl(a'\,\epsilon(\aa)\,a''\bigr) = \epsilon(a')\,\epsilon(\raa) = \epsilon(\taa)
  \end{equation}
  for~\(\taa\in \BB(A,A,A'')\). The elements in the image of~\(\rbb = S\,h_{n-1}(\taa_{1},\dots,\taa_{n})\)
  have length at least~\(1\) for~\(n\ge2\), so that
  \begin{equation}
    \epsilon(f_{n}(\taa_{1},\dots,\taa_{n})) = \epsilon(\ff(a'_{1}\otimes\raa_{1}\cdot \rbb)) = \epsilon(a')\,\epsilon(\raa)\,\epsilon(\rbb)
  \end{equation}
  vanishes. Hence the morphism~\(f\) is augmented.
\end{proof}

\section{The Eilenberg--Moore theorem}
\label{sec:e-m}

We write \(C^{*}(-)\) for normalized singular cochains and \(H^{*}(-)\) for singular cohomology,
both taken with coefficients in~\(\kk\). The augmentations on the cochain complexes are induced
by the inclusion of some chosen base points.

The Eilenberg--Moore theorem describes the cohomology of the pull-back~\(\gE\) of a fibre bundle
or Serre fibration~\(E\to B\) along a map~\(g\colon X\to B\),
\begin{equation}
  \label{eq:pull-back}
  \begin{tikzcd}
    \gE \arrow{d}[left]{\gp} \arrow{r}{\tilde{g}} & E \arrow{d}{p} \\
    X \arrow{r}{g} & B \mathrlap{.}
  \end{tikzcd}
\end{equation}
A homological version appeared in Eilenberg--Moore's original article~\cite[Thm.~12.1]{EilenbergMoore:1966}
for simply connected~\(B\). For field coefficients, the coalgebra structure on~\(H(\gE)\) was also
described there~\cite[Sec.~18]{EilenbergMoore:1966}.
Assuming field coefficients throughout, a cohomological version was given by Smith~\cite[Sec.~I.3]{Smith:1967}.
It is largely reproduced by McCleary~\cite[Sec.~7.1~\&~7.2]{McCleary:2001}.\footnote{%
  Note, however, that there is a mistake in the proof of the Eilenberg--Moore theorem in~\cite{McCleary:2001}: 
  Contrary to what is claimed in the proof of~\cite[Thm.~7.14]{McCleary:2001},
  the filtrations on~\(C^{*}(B)\) and~\(C^{*}(E)\) are not compatible because a (singular!) \(q\)-cochain on~\(B\)
  may not vanish on the \((q-1)\)-skeleton of~\(B\). Presumably, the ``Serre filtration'' used by Smith
  in the proof of~\cite[Prop.~I.3.2]{Smith:1967} did not mean the filtration via skeleta,
  but Serre's original one~\cite[Sec.~II.4]{Serre:1951} applied to singular simplices instead of cubes.
  This would be the same as the ``geometric filtration'' of~\cite[Sec.~11]{EilenbergMoore:1966}.}
Instead of simple connectedness, the version given by Gugenheim--May~\cite[Thm.~3.3, Cor.~3.5]{GugenheimMay:1974}
only assumes the fundamental group of~\(B\) to operate trivially on the cohomology of the fibre of~\(E\to B\).
It also allows more general rings than fields.
In a homological setting, Gugenheim~\cite[Sec.~6.3]{Gugenheim:1972} has given a proof based on the twisted Eilenberg--Zilber theorem,
see also \cite[Prop.~8.3]{Franz:homog} for a similar argument.

All approaches mentioned so far lead to a complex which is shown to be quasi-isomorphic to the (co)chains
on~\(\gE\). This is what we are interested in because we want to apply our \Cref{thm:ainf-bar-twosided}.
There are also other, more geometric proofs of the Eilenberg--Moore theorem that lead directly to a spectral sequence.
See~\cite[Sec.~8.3]{McCleary:2001} for a discussion.

\begin{theorem}
  \label{thm:em-ainf}
  Let \(E\to B\) be a Serre fibration with fibre~\(F\), and let \(\gE\) be the pull-back of~\(E\) via a map~\(g\colon X\to B\).
  Assume the following: \(\kk\) is a principal ideal domain, all spaces and maps are pointed,
  \(B\) and~\(X\) are path-connected, \(\pi_{1}(B)\) acts trivially on~\(H^{*}(F)\), and also one of the following:
  \(H^{*}(F)\) is degreewise finitely generated over~\(\kk\), or both~\(H^{*}(B)\) and~\(H^{n}(X)\) are degreewise finitely generated over~\(\kk\).
  Then the morphism of augmented \(\Ainf\)-algebras
  \begin{equation*}
    f\colon \BB\bigl(C^{*}(X),C^{*}(B),C^{*}(E)\bigr) \Rightarrow C^{*}(\gE)
  \end{equation*}
  (where the target is an augmented dga) is a quasi-isomorphism.
\end{theorem}

Recall from \Cref{sec:diff-tor} that \(\BB(C^{*}(X),C^{*}(B),C^{*}(E))\) computes
the differential torsion product~\(\Tor_{C^{*}(B)}(C^{*}(X),C^{*}(E))\)
since the cochain algebra are torsion-free over the principal ideal domain~\(\kk\).

\begin{proof}
  From~\eqref{eq:pull-back} we get the commutative diagram of augmented dgas
  \begin{equation}
    \begin{tikzcd}
      C^{*}(\gE) & C^{*}(E) \arrow{l}[above]{\tilde{g}^{*}} \\
      C^{*}(X) \arrow{u}{\gp^{*}} & C^{*}(B) \arrow{l}[above]{g^{*}} \arrow{u}[right]{p^{*}} \mathrlap{.}
    \end{tikzcd}
  \end{equation}
  By results of Gugenheim--May~\cite[Ch.~3]{GugenheimMay:1974}, the chain map
  \begin{equation*}
    f_{1}\colon \BB\bigl(C^{*}(X),C^{*}(B),C^{*}(E)\bigr) \to C^{*}(\gE)
  \end{equation*}
  is a quasi-isomorphism, see the discussion in the proof of~\cite[Thm.~A.27]{CarlsonFranz}.
  According to \Cref{thm:augm-ainf}, \(f_{1}\) extends to the morphism~\(f\) of augmented \(\Ainf\)-algebras.
\end{proof}

\begin{remark}
  \label{rem:transfer}
  As mentioned in the introduction, the mere existence of an \(\Ainf\)-structure on the two-sided
  bar construction as well as that of an \(\Ainf\)-extension of the map~\(f_{1}\colon \BB(C^{*}(X),C^{*}(B),C^{*}(E))\to C^{*}(\gE)\)
  follow from standard techniques if one assume field coefficients:
  Since \(f_{1}\) is a homotopy equivalence in this case, \cf~\cite[Prop.~II.4.3]{Dold:1980},
  one can transfer the dga structure on~\(C^{*}(\gE)\) to an \(\Ainf\)-structure
  on the two-sided bar construction and extend \(f_{1}\) to an \(\Ainf\)-morphism, see~\cite{Markl:2006}.
  The result depends on the chosen homotopy inverse~\(k\) to~\(f_{1}\) and on the homotopy
  connecting~\(f_{1}k\) to the identity map on~\(C^{*}(\gE)\). The drawbacks of this approach have been pointed out
  in the introduction already.
\end{remark}

\begin{corollary}
  \label{thm:iterated-bar}
  In addition to the hypotheses of \Cref{thm:em-ainf} assume that the pull-back~\(\gE\) is simply connected.
  Then the cohomology of the iterated bar construction
  \begin{equation*}
    \BB\BB\bigl(C^{*}(X),C^{*}(B),C^{*}(E)\bigr)
  \end{equation*}
  is isomorphic to the cohomology of the loop space~\(\Omega\gE\) as a graded \(\kk\)-module.
\end{corollary}

\begin{proof}
  The bar construction~\(\BB C^{*}(\gE)\) computes the cohomology of~\(\Omega\gE\).
  (This of course follows from \Cref{thm:em-ainf}, applied to the path fibration of~\(\gE\).)
  We get a quasi-iso\-mor\-phism between this bar construction and the iterated one displayed above
  by applying \(\BB\) to the \(\Ainf\)-quasi-isomorphism~\(f\) from \Cref{thm:em-ainf}.
  That \(\BB f\) is again a quasi-isomorphism can shown by the same argument
  as for dga maps: The filtration of each bar construction by the length of elements
  is exhaustive and leads to a cohomological spectral sequence in the left half-plane.
  The map induced by~\(\BB f\) between the first pages is an isomorphism, hence so is \(H^{*}(\BB f)\).
\end{proof}

\begin{remark}
  \label{rem:fresse}
  If \(A\) is a cdga, then so is its bar construction~\(\BB A\). Hence one can form the \(n\)-th iterated bar construction~\(\BB^{n}A\)
  for any~\(n\ge0\). Fresse has shown that one can extend the functor~\(A\mapsto\BB^{n}A\) to the category of~\(E_{n}\)-algebras,
  but not further \cite[p.~751 \&~Thm.~5.5]{Fresse:2011}. 
  An \(E_{n}\)-algebra is an algebra over a dg~operad equivalent
  to the chain operad of little \(n\)-cubes. Since hgas are a special kind of \(E_{2}\)-algebras (\cf~\cite[\S 1.6.6]{BergerFresse:2004}),
  \Cref{thm:iterated-bar} is in line with Fresse's result.
\end{remark}

\section{More preliminaries}
\label{sec:prelim-2}

\subsection{Shc algebras}

Let \(A\) be an augmented dga. We say that \(A\) is a \newterm{strongly homotopy commutative (shc)} algebra
if the multiplication map~\(\mu_{A}\colon A\otimes A\to A\)
extends to an \(\Ainf\)-map~\(\Phi\colon A\otimes A\Rightarrow A\) in the sense that \(\Phi_{1}=\mu_{A}\).
There are stronger definitions of an shc algebra (see~\cite[Sec.~1]{Franz:hgashc}), but we will not need those.

Assume that \(A\) is an hga.
A natural shc structure~\(\Phihga\) on~\(A\) was given in~\cite[Prop.~4.1]{Franz:hgashc} by the formula
\begin{equation}
  \Phihga_{n} = (-1)^{n-1}\sum_{\substack{j_{1}+\dots+j_{n}\\=n-1}} \Phi_{\jj}
\end{equation}
where the sum is over all decompositions~\(\jj=(j_{1},\dots,j_{n})\) of~\(n-1\) into \(n\)~non-neg\-a\-tive integers 
such that
\begin{equation}
  \label{eq:cond-j}
  \forall\, 1\le s\le n \qquad j_{1}+\dots+j_{s} < s,
\end{equation}
and
\begin{multline}
  \label{eq:def-Phihga}
  \Phi_{\jj}(a_{1}\otimes b_{1},\dots,a_{n}\otimes b_{2}) \eqKS
  E_{j_{1}}(a_{1};b_{1},\dots,b_{j_{1}}) \\* \cdot E_{j_{2}}(a_{2};b_{j_{1}+1},\dots,b_{j_{1}+j_{2}})
  \cdots E_{j_{n}}(a_{n};b_{n-j_{n}},\dots,b_{n-1})\,b_{n}
\end{multline}
for elements~\(a_{1}\otimes b_{1}\),~\dots,~\(a_{n}\otimes b_{n}\in A\otimes A\).
The condition~\eqref{eq:cond-j} means that the variables~\(b_{j}\) appearing in an \(E\)-term in~\eqref{eq:def-Phihga}
together with a variable~\(a_{i}\) have indices~\(j<i\).
It implies \(j_{1}=0\), so that the leading factor in~\eqref{eq:def-Phihga} is simply \(E_{0}(a_{1})=a_{1}\).

As shown in~\cite[Prop.~4.3]{Franz:hgashc},
this shc structure has the property that its composition with the shuffle map~\eqref{eq:def-shuffle},
\begin{equation}
  \BB A\otimes \BB A \xrightarrow{\shuffle} \BB(A\otimes A) \xrightarrow{B\Phihga} \BB A,
\end{equation}
is the product in~\(\BB A\) corresponding to the hga structure on~\(A\).

\begin{remark}
  \label{rem:m2-hga}
  In terms of this shc structure, the Carlson--Franz product~\(m_{2}\) on a two-sided bar construction~\(\BB(A',A,A'')\)
  of hgas can be given as
  \begin{equation}
    \label{eq:def-m2-orig}
    m_{2} = \BB(\Phihga,\Phihga,\Phihga)\,\shuffle,
  \end{equation}
  where \(\BB(\Phihga,\Phihga,\Phihga)\) is a map of the form~\eqref{eq:ainf-twosidedbar}  
  and \(\shuffle\) the shuffle map~\eqref{eq:def-shuffle-twosided}.
  This is in fact its original definition of~\(m_{2}\), see~\cite[eq.~(A.3)]{CarlsonFranz}.
  The symmetry between~\(A'\) and~\(A''\) in the expression~\eqref{eq:def-m2-orig} is lost
  in the explicit formula~\eqref{eq:def-m2} for~\(m_{2}\). This is caused by the asymmetrical nature of the condition~\eqref{eq:EE-length-2}
  for the twisting cochain~\(\EE\), leading to an asymmetry between the arguments of~\(\Phihga\).
\end{remark}

\subsection{The transpose}

We write the dual of a complex~\(A\) as~\(A^{\vee}=\Hom(A,\kk)\).
If \(A\) is a (non-negatively graded) homological complex, then \(A^{\vee}\) is a (non-positively graded) cohomological one.
For example, the dual~\(C(X)^{\vee}\) of the normalized singular chain complex~\(C(X)\) of a space~\(X\)
is the normalized singular cochain complex~\(C^{*}(X)\).

For any map~\(f\colon A\to B\) between complexes we define its transpose
\begin{equation}
  f^{*}\colon B^{\vee} \to A^{\vee},
  \qquad
  f^{*}(\beta) = (-1)^{\deg{f}\deg{\beta}}\,\beta\,f. 
\end{equation}
Comparison with~\eqref{eq:diff-map} shows
\begin{equation}
  \label{eq:transp-diff}
  d_{C^{\vee}} = - (d_{C})^{*}.
\end{equation}

Given another map~\(g\colon B\to C\), one has
\begin{equation}
  \label{eq:transp-swap}
  (g\,f)^{*} = (-1)^{\deg{g}\deg{f}}\,f^{*} g^{*}.
\end{equation}
Moreover, for maps~\(f_{i}\colon A_{i}\to B_{i}\), \(i=1\),~\(2\), the following diagram commutes:
\begin{equation}
  \label{eq:transp-tensor}
  \begin{tikzcd}
    B_{1}^{\vee}\otimes B_{2}^{\vee} \arrow{d}{f_{1}^{*}\otimes f_{2}^{*}} \arrow{r} & (B_{1}\otimes B_{2})^{\vee} \arrow{d}{(f_{1}\otimes f_{2})^{*}} \\
    A_{1}^{\vee}\otimes A_{2}^{\vee} \arrow{r} & (A_{1}\otimes A_{2})^{\vee} \mathrlap{.}
  \end{tikzcd}
\end{equation}

The dual~\(C^{\vee}\) of a (co-augmented) dgc~\(C\) is an (augmented) dga with multiplication
\begin{equation}
  \mu_{C^{\vee}}\colon C^{\vee}\otimes C^{\vee} \to (C\otimes C)^{\vee} \xrightarrow{\Delta^{*}} C^{\vee}.
\end{equation}
Similarly, the transpose of a morphism~\(f\colon C\to D\) of (co-augmented) dgcs is a morphism of (augmented) dgas.

We fix the isomorphism of complexes~\((\desusp A)^{\vee}\cong\susp A^{\vee}\) that makes the diagram
\begin{equation}
  \begin{tikzcd}
    (\desusp A)^{\vee} \arrow{rr}{\cong} \arrow{dr}[below left,pos=0.4]{(\desusp)^{*}} & & \susp A^{\vee} \arrow{dl}{\desusp} \\
    & A^{\vee}
  \end{tikzcd}
\end{equation}
commute. Under this isomorphism we have
\begin{equation}
  \label{eq:transp-susp-desusp}
  (\desusp_{A})^{*} = \desusp_{A^{\vee}}
  \qquad\text{and}\qquad
  (\susp_{A})^{*} = -\susp_{A^{\vee}}
\end{equation}
where the second identity follows from~\eqref{eq:transp-swap}.
(In particular, it is not possible to have a~``\(+\)'' in both equations.)
This choice of signs will be important in what follows.

\subsection{The cobar construction}
\label{sec:cobar}

Let \(C\) be a coaugmented dgc with augmentation ideal~\(\bar C\). The decomposition~\(C=\bar C\oplus\kk\)
gives a canonical projection~\(\pi=\pi_{C}\colon C\twoheadrightarrow\bar C\) besides the inclusion~\(\eta=\eta_{C}\colon\bar C\hookrightarrow C\).

The cobar construction of~\(C\) is the augmented dga
\begin{equation}
  \Cobar C = \bigoplus_{k\ge0} \Cobar_{k}C
  \qquad\text{with}\qquad
  \Cobar_{k}C = (\desusp\bar C)^{\otimes k} \,;
\end{equation}
elements of~\(\Cobar_{k}C\) are written in the form~\(\CobarEl{c_{1}|\dots|c_{k}}\).
Multiplication is concatenation of elements; the unit element is~\(\CobarEl{}=\BBone\in\Cobar_{0}C=\kk\).
The augmentation~\(\epsilon\colon\Cobar C\to\kk\) is the identity on~\(\Cobar_{0}C=\kk\) and zero
on the components~\(\Cobar_{k}C\) with~\(k>0\).

The differential on~\(\Cobar C\) is the unique derivation given on~\(\Cobar_{1}C\) by
\begin{equation}
  d\CobarEl{c} = -\CobarEl{dc} + (\desusp\otimes\desusp)\,\bar\Delta\,c
\end{equation}
where
\begin{equation}
  \bar\Delta\colon \bar C \hookrightarrow C  \xrightarrow{\Delta} C \otimes C \twoheadrightarrow \bar C\otimes\bar C
\end{equation}
is the reduced diagonal.

The canonical map
\begin{equation}
  \BB(C^{\vee}) \to (\Cobar C)^{\vee}
\end{equation}
given by the component maps
\begin{equation}
  B_{k}(C^{\vee}) = (\susp\,\bar C^{\vee})^{\otimes k} \cong \bigl((\desusp \bar C)^{\vee}\bigr)^{\otimes k}
  \to \bigl((\desusp \bar C)^{\otimes k}\bigr)^{\vee} = (\Cobar_{k} C)^{\vee}
\end{equation}
is a morphism of coaugmented dgcs. To see that we get a chain map, we note that
by~\eqref{eq:transp-swap},~\eqref{eq:transp-tensor} and~\eqref{eq:transp-susp-desusp} we have
\begin{equation}
  \bigl((\desusp\otimes\desusp)\,\bar\Delta\,\susp\bigr)^{*}
  = \susp^{*}\,\bar\Delta^{*}\,\bigl((\desusp)^{*}\otimes(\desusp)^{*}\bigr)
  = -\susp\,\bar\mu\,(\desusp\otimes\desusp),
\end{equation}
where \(\bar\mu\colon\colon \bar C^{\vee}\otimes \bar C^{\vee}\to \bar C^{\vee}\)
denotes the reduced multiplication map of the augmented dga~\(C^{\vee}\).
This expression is the negative of the part~\(B_{2}C^{\vee}\to B_{1}C^{\vee}\) of the differential on~\(BC^{\vee}\),
as required by the identity~\eqref{eq:transp-diff}.

\subsection{\texorpdfstring{\(\Ainf\)}{A∞}-morphisms between dgcs}

Let \(C\) and~\(D\) be coaugmented dgcs, and let
\begin{equation}
  G\colon \Cobar C \to \Cobar D
\end{equation}
be a morphism of augmented dgas.
This morphism is determined by the maps
\begin{multline}
  \qquad
  g_{n}\colon C \stackrel{\pi}{\twoheadrightarrow} \bar C \xrightarrow{\desusp} \desusp \bar C = \Cobar_{1}C \hookrightarrow \Cobar C \\
  \xrightarrow{G} \Cobar D \twoheadrightarrow \Cobar_{n} D = (\desusp \bar D)^{\otimes n} \xrightarrow{\susp^{\otimes n}} \bar D^{\otimes n} \xrightarrow{\eta^{\otimes n}} D^{\otimes n}
  \qquad
\end{multline}
of degree~\(n-1\) for~\(n\ge2\). The map~\(g_{1}\colon C\to D\) is defined analogously, but with the map~\(\iota_{D}\,\epsilon_{C}\colon C\to D\) added,
\begin{equation}
  g_{1}(c) = \susp\, G(\CobarEl{\bar c}) + \iota_{D}\,\epsilon_{C}(c).
\end{equation}
Together this family of maps forms an \(\Ainf\)-coalgebra morphism~\(g\colon C\Rightarrow D\), and we write \(G\) as~\(\Cobar g\colon \Cobar C\to\Cobar D\).

The components of~\(g\) satisfy the identity
\begin{equation}
  d(g_{n}) = \sum_{i=1}^{n-1}(-1)^{i}\,(g_{i}\otimes g_{n-i})\,\Delta
  +\sum_{i=1}^{n-2}(-1)^{i}\,\bigl(1^{\otimes i}\otimes\Delta\otimes 1^{n-i-2}\bigr)\,g_{n-1}
\end{equation}
for all~\(n\ge1\). Moreover, \(g_{1}(1_{C})=1_{D}\) and \(g_{1}(\bar C)\subset\bar D\).
For~\(n\ge2\) we have \(g_{n}(1_{C})=0\) and the image of~\(g_{n}\) lies in~\(\bar C^{\otimes n}\subset C^{\otimes n}\).
Since \(\Cobar C\) is the direct sum of the terms~\(\Cobar_{k}C\), we also have that for any~\(c\in C\) only finitely many terms~\(g_{n}(c)\) are non-zero.
Conversely, any family of maps
\begin{equation}
  g_{n}\colon C\to D^{\otimes n}
\end{equation}
of degree~\(n-1\) for~\(n\ge1\) having these properties defines an \(\Ainf\)-coalgebra morphism~\(g\colon C\Rightarrow D\).

Given an \(\Ainf\)-coalgebra morphism~\(g\colon C\Rightarrow D\), the transposes
\begin{align}
  (g_{n})^{*} &= \bigl(\eta^{\otimes n}\,\susp^{\otimes n}\,\Cobar g\,\desusp\,\pi\bigr)^{*}
  = (-1)^{n}\,\pi^{*}\,(\desusp)^{*}\,(\Cobar g)^{*}\,(\susp^{\otimes n})^{*}\,(\eta^{*})^{\otimes n} \\
  \notag &= \eta_{C^{\vee}}\,\desusp\,(\Cobar g)^{*}\,\susp^{\otimes n}\,(\pi_{D^{\vee}})^{\otimes n}
\end{align}
of the components of~\(g\)
are the components~\((g^{*})_{n}\) of a strictly unital \(\Ainf\)-algebra morphism~\(g^{*}\colon D^{\vee}\Rightarrow C^{\vee}\).
(Here we are again using the formulas~\eqref{eq:transp-swap},~\eqref{eq:transp-tensor} and~\eqref{eq:transp-susp-desusp},
and we have omitted the separate definition for~\(n=1\).)
Moreover, we obtain a commutative diagram
\begin{equation}
  \begin{tikzcd}
    \BB(D^{\vee}) \arrow{d}{\BB(g^{*})} \arrow{r} & (\Cobar D)^{\vee} \arrow{d}{(\Cobar g)^{*}} \\
    \BB(C^{\vee}) \arrow{r} & (\Cobar C)^{\vee}\mathrlap{.}
  \end{tikzcd}
\end{equation}

In analogy with the algebra case,
we define an \newterm{shc coalgebra structure} on a coaugmented dgc~\(C\) to be an \(\Ainf\)-coalgebra map
\begin{equation}
  \Psi\colon C \Rightarrow C\otimes C
\end{equation}
whose base component~\(\Psi_{1}\colon C\to C\otimes C\) is the diagonal of~\(C\).

\subsection{Homotopy Gerstenhaber coalgebras}

Let \(C\) be a coaugmented dgc and assume that the augmented dga~\(\Cobar C\) has the structure of a dg~bialgebra,
that is, comes with a diagonal
\begin{equation}
  \Delta_{\Cobar C}\colon \Cobar C \to \Cobar C\otimes\Cobar C
\end{equation}
that is compatible with the dga structure and the augmentation. It is determined by its associated twisting cochain
\begin{equation}
  \EE\colon C \twoheadrightarrow \bar C \xrightarrow{\desusp} \desusp\bar C = \Cobar_{1} C \xrightarrow{\Delta_{\Cobar C}} \Cobar C\otimes\Cobar C,
\end{equation}
which in turn in determined by its compositions with
\begin{multline}
  \qquad
  \rho_{kl}\colon \Cobar C\otimes\Cobar C \twoheadrightarrow \Cobar_{k} C\otimes\Cobar_{l} C = (\desusp\bar C)^{\otimes k} \otimes (\desusp\bar C)^{\otimes l} \\
  \xrightarrow{\susp^{\otimes(k+l)}} \bar C^{\otimes(k+l)} \hookrightarrow C^{\otimes(k+l)}.
  \qquad
\end{multline}
for~\(k\),~\(l\ge0\).
We call \(C\) a \newterm{homotopy Gerstenhaber coalgebra (hgc)} if the composition~\(E^{kl}=\rho_{kl}\,\EE\) vanishes unless \((k,l)=(0,1)\) or~\(k=1\).
One always has \(E^{01}(c)=E^{10}(c)=\bar c\).
For any~\(k\ge0\) we define the \newterm{homotopy Gerstenhaber cooperation}
\begin{equation}
  E^{k}\colon C \to C\otimes C^{\otimes k},
  \qquad
  E^{k}(c) = \begin{cases}
    E^{1k}(c) & \text{if \(k>0\).} \\
    c & \text{if \(k=0\).}
  \end{cases}
\end{equation}
Their transposes define an hga structure on~\(C^{\vee}\), \cf~\cite[Rem.~3.2]{Franz:szczarba2}.

\subsection{Simplicial sets}
\label{sec:simplicial}

The normalized chain complex~\(C(X)\) of any simplicial set~\(X\) is naturally an hgc,
for any coaugmentation~\(\iota\colon\kk\to C(X)\). This includes the normalized
singular chain complexes of topological spaces. It is convenient to define
the hgc operations in terms of interval cut operations, which we briefly recall first.
See \cite[Secs.~1.2, 2.2]{BergerFresse:2004} for a comprehensive treatment.

We call \([i,j]=\{i,i+1,\dots,j\}\) with~\(0\le i\le j\) an interval and write~\([j]=[1,j]\).
Let \(u\colon[r+k]\to[r]\) be a surjection, \(r\ge0\), \(k\ge0\).
We assume \(u\) to be non-degenerate in the sense that \(u(t)\ne u(t+1)\) for all~\(1\le t<r+k\).
The associated interval cut operation for a simplicial set~\(X\),
\begin{equation}
  \AWu{u}\colon C(X) \to C(X)^{\otimes r},
  \qquad x \mapsto \AWu{u}(x),
\end{equation}
is of degree~\(k\) and defined as follows. Let \(x\in X\) be a simplex of dimension~\(m\). For a subdivision
\begin{equation}
  0 = m_{0} \le m_{1} \le \dots \le m_{n+k} = m
\end{equation}
of the interval~\([0,m]\), we define the simplex~\(x_{j}\in X\) with~\(1\le j\le r\) to be the restriction of~\(x\)
to the concatenation of all intervals~\([m_{t},m_{t+1}]\) with label~\(u(t)=j\). (If in this concatenation
the right endpoint of one interval coincided with the left endpoint of the next interval,
then this common endpoint is repeated, forcing the resulting simplex to be degenerate.) Then
\begin{equation}
  \AWu{u}(x) = \sum \pm x_{1}\otimes \dots \otimes x_{r}
\end{equation}
where the sum is over all such decompositions. The sign left unspecified above
is the product of a permutation sign and a position sign, see~\cite[\S 2.2.4]{BergerFresse:2004} for details.

In this language, the hgc cooperations are given by the interval cut operations
\begin{equation}
  \label{eq:def-hg-chains}
  E^{k} = \AWu{(1,2,1,3,1,\dots,1,k+1,1)},
\end{equation}
and their transposes~\((E^{k})^{*}\) are the operations~\(E_{k}\) for the hga~\(C^{*}(X)\), \cf~\cite[Sec.~8.2]{Franz:homog}.
Note that when applying \(E^{k}\) to a simplex,
the intervals labelled \(2\),~\dots,~\(k+1\) must have length at least~\(1\)
for otherwise the simplex corresponding to the label~\(1\) would be degenerate.
Hence \(E^{k}\) vanishes on (normalized) chains of degree~\(<k\).

Let us write \(\mu^{(n)}\) and~\(\Delta^{(n)}\) for the \(n\)-fold iteration of the product and coproduct, respectively.
Since \(\Phi_{\jj}\) is, up to a reordering of the tensor factors in the argument, equal to
\begin{equation}
  \mu^{(n+1)}\bigl(E_{j_{1}}\otimes\dots\otimes E_{j_{n}}\otimes 1\bigr),
\end{equation}
we have \(\Phi_{\jj}=\AWu{u(\jj)}^{*}\) for the surjection
\begin{multline}
  \label{eq:u-a-b}
  \qquad
  u(\jj) = (v_{1},w_{1},v_{1},\dots,v_{1},w_{j_{1}},v_{1}, \\
  v_{2},w_{j_{1}+1},v_{2},\dots,v_{2},w_{j_{1}+j_{2}},v_{2}, \\
  \dots,v_{n}, w_{n-j_{n}},v_{n},\dots,v_{n},w_{n-1},v_{n},w_{n})
  \qquad
\end{multline}
of length~\(3n-1\) and degree~\(n-1\),
where we have set \(v_{t}=2t-1\) and~\(w_{t}=2t\) for~\(1\le t\le n\). For example,
\begin{equation}
  u((0,0,2,1)) = (1,3,5,2,5,4,5,7,6,7,8).
\end{equation}

Hence, if we define
\begin{equation}
  \Psihgc_{n} = (-1)^{n-1} \sum_{\substack{j_{1}+\dots+j_{n}\\=n-1}} \Psi_{\jj}
  \qquad\text{with}\qquad
  \Psi_{\jj}=\AWu{u(\jj)},
\end{equation}
we have
\begin{equation}
  \Phi_{\jj} = \Psi_{\jj}^{*}
  \qquad\text{and}\qquad
  \Phihga_{n} = (\Psihgc_{n})^{*}.
\end{equation}
As remarked above, the surjection~\(u(\jj)\) associated to a decomposition~\(\jj\) of~\(n-1\) has degree~\(n-1\).
This implies that \(\Psi_{\jj}\) vanishes on simplices of degree~\(<n-1\), and so does \(\Psihgc_{n}\).
The maps~\(\Psihgc_{n}\) therefore assemble to an shc coalgebra structure
\begin{equation}
  \Psihgc\colon C(X) \Rightarrow C(X)\otimes C(X)
\end{equation}
whose dual  is the \(\Ainf\)-algebra structure~\(\Phihga\) on the hga~\(C^{\vee}\) given by~\eqref{eq:def-Phihga}.\footnote{%
  Defining \(\Psi_{\jj}\) directly in terms of the hga cooperations gives
  an \(\Ainf\)-coalgebra structure on any hgc if one assume that \(E^{k}(c)\) vanishes for~\(\deg{c}<k\).
  The proof is analogous to the hga case carried out in~\cite[Prop.~4.1]{Franz:hgashc}.}
(If \(\Psihgc\) were not an \(\Ainf\)-coalgebra map, then its transpose would fail to be an \(\Ainf\)-algebra map
because sending a map between free complexes to its transpose is injective.)

\subsection{Shc structures via contractions}
\label{sec:contractions}

Let \(C\) and~\(D\) be complexes. A \newterm{contraction} is a triple~\((f,g,h)\),
\begin{equation}
  C \mathop{\leftrightarrows}^{g}_{f} D \; \selfmapright \, h
\end{equation}
where \(f\colon C\to D\) and~\(g\colon D\to C\) are chain maps
and \(h\colon D\to D\) a homotopy such that
\begin{equation}
  g\,f = 1_{C},
  \qquad
  d(h) = f\,g - 1_{D}
  \qquad
  h\,f = 0,
  \qquad
  g\,h = 0,
  \qquad
  h\,h = 0.
\end{equation}
The latter three identities are often called ``side conditions''.

Now assume additionally that \(C\) and~\(D\) are coaugmented dgcs and \(f\) a morphism between them.
In this setting, Gugenheim--Munkholm~\cite[Thm~4.1\(^{*}\)]{GugenheimMunkholm:1974} have given a recursive formula
for an \(\Ainf\)-coalgebra map~\(G\colon D\Rightarrow C\) extending~\(g\), that is, with
\begin{equation}
  G_{1}=g.
\end{equation}
Their recursive formula for~\(n\ge2\) is\footnote{%
  The sign~\((-1)^{n-1}\) in~\eqref{eq:GM-rec-def} is implicit in the formula given in~\cite[p.~26]{GugenheimMunkholm:1974}
  because the map~\(\bar{F}_{p}\) there is, in our notation, a map to \((\desusp\bar{C})^{\otimes p}\) of degree~\(-1\) ,
  not one to~\(\bar C^{\otimes p}\) of degree~\(p-1\).}
\begin{equation}
  \label{eq:GM-rec-def}
  G_{n} = -\sum_{l=1}^{n-1} (-1)^{n-l}\,(G_{l}\otimes G_{n-l})\,\Delta_{D}\,h.
\end{equation}
Stated in this generality, one needs to assume that the finiteness condition for \(\Ainf\)-coalgebra maps holds,
meaning that for each~\(d\in D\), \(G_{n}(d)\) vanishes for large enough~\(n\).
In addition to extending~\(g\), the \(\Ainf\)-coalgebra map~\(G\) is a homotopy inverse to~\(f\).
More precisely, Gugenheim--Munkholm show that
\begin{equation}
  \label{eq:f-G-id}
  f\circ G = 1_{C}
\end{equation}
holds and that \(G\circ f\) is homotopic to~\(1_{D}\) via an explicit \(\Ainf\)-coalgebra homotopy
having certain additional properties. In what follows, we will only need \eqref{eq:f-G-id}.

The Gugenheim--Munkholm construction applies in particular to the Eilenberg--Zilber contraction
\begin{equation}
  \label{eq:ez-contraction}
  C(X)\otimes C(Y) \mathop{\leftrightarrows}^{\AW}_{\shuffle} C(X\times Y) \, \selfmapright \, h
\end{equation}
for two spaces or simplicial sets~\(X\) and~\(Y\).
Here \(f=\shuffle\) is the shuffle map, \(g=\AW\) the Alexander--Whitney map and \(h\) the Eilenberg--Mac\,Lane homotopy,
see for instance~\cite[Sec.~3]{Franz:szczarba1}.
In fact, there are two canonical homotopies one can use. One is the original one
recursively defined by Eilenberg--Mac\,Lane~\cite[eqs.~(2.13)]{EilenbergMacLane:1954};
a non-recursive definition was much later given by Rubio and Morace~\cite[Sec.~3.1]{Rubio:1991}.
An equally good choice is the ``opposite'' homotopy
that exchanges the roles of~\(X\) and~\(Y\) and also those of front on back faces of simplices,
see~\cite[Sec.~3.3]{Rubio:1991},~\cite[Sec.~4]{Franz:szczarba1}.
We are going to use this opposite homotopy, which is defined by
\begin{equation}
  \label{eq:def-h}
  h(x,y) = \!\!\! \sum_{\substack{0\le p+q < n\\(\alpha,\beta)\vdash(p,q+1)}} \!\!\! (-1)^{p+q+(\alpha,\beta)}\,
  \bigl(s_{\beta}\,\partial_{p+1}^{p+q}\,x, s_{p+q+1}\,s_{\alpha}\,\partial_{0}^{p-1}\,y\bigr)
\end{equation}
for~\((x,y)\) in~\(X_{n}\times Y_{n}\), \cf~\cite[eq.~(4.8)]{Franz:szczarba1}.
Here \((\alpha,\beta)\vdash(p,q+1)\) means that \((\alpha,\beta)\) is a \newterm{\((p,q+1)\)-shuffle}, that is,
a partition of the set~\(\{0,\dots,p+q\}\) with \(p\)~elements in~\(\alpha\) and \(q+1\) in~\(\beta\).
Its signature is denoted by~\((-1)^{(\alpha,\beta)}\).
Moreover,
\begin{equation}
  s_{\alpha} = s_{\alpha_{p}} \circ \dots \circ s_{\alpha_{1}}
\end{equation}
for~\(\alpha=\{\alpha_{1}<\dots<\alpha_{p}\}\) is a repeated degeneracy operator and
\begin{equation}
  \partial_{i}^{j} = \partial_{i} \circ \dots \circ \partial_{j}
\end{equation}
for~\(j\le i-1\) a repeated face operator. Note that \(s_{\emptyset}\) and~\(\partial_{i}^{i-1}\) are identity maps.

In the context of the Eilenberg--Zilber contraction
we write the \(\Ainf\)-coalgebra map defined by~\eqref{eq:GM-rec-def} as
\begin{equation}
  \label{eq:GM-XY}
  G = G_{X,Y}\colon C(X\times Y) \Rightarrow C(X)\otimes C(Y).
\end{equation}
That the finiteness condition holds in this case will be a consequence
of our proof of \Cref{thm:Phihga-PhiGM} in the next section.

\section{Comparison of the two shc structures on cochains}
\label{sec:compare-shc}

Let \(\Delta=\Delta_{X}\colon X\to X\times X\) be the diagonal.
Since the composition
\begin{equation}
  C^{*}(X)\otimes C^{*}(X) \to \bigl(C(X)\otimes C(X)\bigr)^{\vee} \xrightarrow{\AW^{*}} C^{*}(X\times X) \xrightarrow{\Delta^{*}} C^{*}(X)
\end{equation}
is the definition of the cup product in~\(C^{*}(X)\), we conclude that the transpose~\(G_{X,X}^{*}\) of
Gugenheim--Munkholm's \(\Ainf\)-coalgebra map~\eqref{eq:GM-XY}
defines an shc algebra structure~\(\PhiGM=\PhiGM_{X}\) on~\(C^{*}(X)\) via the composition\footnote{\label{foot:triv-ext}
  The same construction appears in Munkholm's separate work~\cite[Prop.~4.7]{Munkholm:1974}, based on the transpose
  of the recursive formula~\eqref{eq:GM-rec-def}, \cf~the proof of~\cite[Prop.~2.2]{Munkholm:1974}.
  Note that the homotopy in Munkholm's cohomological contraction (called a `trivialized extension')
  is~\(-h^{*}\) in our notation.}
\begin{equation}
  \label{eq:def-PhiGM}
  \PhiGM_{X}\colon
  C^{*}(X)\otimes C^{*}(X) \to \bigl(C(X)\otimes C(X)\bigr)^{\vee} \xRightarrow{G_{X,X}^{*}} C^{*}(X\times X) \xrightarrow{\Delta^{*}} C^{*}(X).
\end{equation}
(No finiteness condition is needed for cochains, see~\cite[\S 4.1\(_{*}\)]{GugenheimMunkholm:1974}.)

\begin{theorem}
  \label{thm:Phihga-PhiGM}
  Given any simplicial set~\(X\),
  the shc structure on~\(C^{*}(X)\) due to Gugenheim--Munkholm (with homotopy~\eqref{eq:def-h})
  and the one for general hgas agree,
  \begin{equation}
     \PhiGM = \Phihga \colon C^{*}(X)\otimes C^{*}(X) \Rightarrow C^{*}(X).
  \end{equation}
\end{theorem}

\begin{corollary}
  \label{thm:mult-bar-GM}
  The composition
  \begin{equation*}
    \BB C^{*}(X) \otimes \BB C^{*}(X) \xrightarrow{\shuffle} \BB\bigl(C^{*}(X)\otimes C^{*}(X)\bigr)
    \xrightarrow{\PhiGM} \BB C^{*}(X\times X) \xrightarrow{\BB\Delta^{*}} \BB C^{*}(X)
  \end{equation*}
  is the multiplication on~\(\BB C^{*}(X)\) given by the hga structure of~\(C^{*}(X)\).
\end{corollary}

For a \(1\)-reduced simplicial set~\(X\), this latter result was established
by Hess--Parent--Scott--Tonks~\cite[Sec.~5]{HessEtAl:2006}, as least in characteristic~\(2\).\footnote{%
  All signs are ignored in~\cite[Sec.~5]{HessEtAl:2006} ``in the interest of simplifying the notation'' \cite[p.~870]{HessEtAl:2006}.}

The remainder of this section is devoted to a proof of \Cref{thm:Phihga-PhiGM}.
We will deduce the claim from the equivalent identity
\begin{align}
  \label{eq:id-cohom}
  G_{X,Y}^{*} = \Phihga_{X\times Y}\circ\bigl(p_{X}^{*}\otimes p_{Y}^{*}) \colon & C^{*}(X)\otimes C^{*}(Y) \Rightarrow C^{*}(X\times Y) \\
\shortintertext{for all simplicial sets~\(X\),~\(Y\), or rather from the homological identity}
  \label{eq:id-hom}
  G_{X,Y} = (p_{X}\otimes p_{Y})\circ \Psihgc_{X\times Y} \colon & C(X\times Y) \Rightarrow C(X)\otimes C(Y)
\end{align}
of which it is the transpose.

\def\GG{\tilde G}
\def\n{m}
Let \(\n\ge0\). For a partition~\((\alpha,\beta)\vdash(p,q+1)\) with~\(p\),~\(q\in[0,\n]\) we define
\begin{align}
  \label{eq:def-psi}
  h_{\alpha,\beta}\colon X_{\n-q}\times Y_{\n-p} &\to C_{\n+1}(X\times Y), \\
  \notag
  (x,y) &\mapsto (s_{\beta}\,x, s_{p+q+1}\,s_{\alpha}\,y)
\end{align}
so that
\begin{equation}
  h(x,y) = \!\!\! \sum_{\substack{0\le p+q < \n\\(\alpha,\beta)\vdash(p,q+1)}} \!\!\! (-1)^{p+q+(\alpha,\beta)}\,
  h_{\alpha,\beta}\bigl(\partial_{p+1}^{p+q}\,x,\partial_{0}^{p-1}\,y\bigr)
\end{equation}
for~\((x,y)\in X_{\n}\times Y_{\n}\). We also write
\begin{equation}
  \GG_{n} = (p_{X}\otimes p_{Y})^{\otimes n}\,\Psihgc_{n}\colon C(X\times Y) \to \bigl(C(X)\otimes C(Y)\bigr)^{\otimes n}.
\end{equation}
for~\(n\ge1\). Translating the identity~\eqref{eq:id-hom} to this notation, we want to prove
\begin{equation}
  \label{eq:GGn-Gn}
  G_{n} = \GG_{n}
\end{equation}
for all~\(n\ge1\).
Because the components of~\(\Psihgc\) satisfy the finiteness condition for \(\Ainf\)-coalgebra maps,
the identity~\eqref{eq:GGn-Gn}, once established, will in particular show that the family~\(G_{n}\) also satisfies this condition
and therefore defines an \(\Ainf\)-coalgebra map, as asserted in \Cref{sec:contractions}.

We introduce some more notation and terminology. Let \((\alpha,\beta)\) be a \((p,q+1)\)-shuffle.
Then \(\beta\ne\emptyset\), and it can be decomposed into intervals,
that is, written as a disjoint union of \(r\ge1\)~intervals
\begin{equation}
  \beta = [i_{1},j_{1}]\cup[i_{2},j_{2}]\cup\dots\cup[i_{r},j_{r}],
\end{equation}
where
\begin{equation}
  0 \le i_{1} < j_{1}+1 < i_{2} < j_{2}+1 < \dots < i_{r} < j_{r}+1 \le p+q+1.
\end{equation}

Let \(x\) be an \(\n\)-simplex,
and let \(J\subset[0,\n]\) be non-empty.
We write \(x(J)\) for the restriction of~\(x\) to~\(J\), that is,
for \(\partial_{i_{k}}\dots\partial_{i_{1}}x\) where \(\{i_{1}<\dots<i_{k}\} = [0,\n]\setminus J\).
If \(J=[i,j]\) is an interval, we simply write \(\restr{x}{i}{j}\).
Moreover, we say that \(x\) is \newterm{degenerate on the interval~\([i,j+1]\)} if there is a simplex~\(\tilde{x}\)
such that \(x=s_{[i,j]}\tilde{x}\). (This condition is defined to be void for~\(i=j+1\).)

The following result is related to~\cite[Thm.~5.1, Lemma~5.4]{HessEtAl:2006}.

\def\si{i}

\begin{lemma}
  \label{thm:GkGl}
  Let \(\n\ge0\), and let \((\alpha,\beta)\vdash(p,q+1)\) with~\(p\),~\(q\in[0,\n]\).
  For any~\((x,y)\in X_{\n-q}\times Y_{\n-p}\) and any~\(n\ge1\) we have the following:
  \begin{enumroman}
  \item
    \label{thm:GkGl-1}
    For any~\(1< k<n\),
    \begin{equation*}
      (\GG_{k}\otimes 1)\,\Delta\,h_{\alpha,\beta}(x,y) = 0.
    \end{equation*}
  \item
    \label{thm:GkGl-2}
    If \(\beta\) is not an interval, then
    \begin{equation*}
      (\AW\otimes \GG_{n-1})\,\Delta\, h_{\alpha,\beta}(x,y) = 0.
    \end{equation*}
  \item
    \label{thm:GkGl-3}
    If \(\beta=[\si,\si+q]\) is an interval, then
    \begin{multline*}
      \quad
      (\AW\otimes\GG_{n-1})\,\Delta\, h_{\alpha,\beta}(x,y) =  (-1)^{n(\si+q+1)}\,\bigl(\restr{x}{0}{\si} \otimes \restr{y}{0}{q+1}\bigr) \\
      \otimes \GG_{n-1}\bigl(\restr{x}{\si}{\n-q}, s_{[0,p-\si]}\,\restr{y}{q+1}{\n-p}\bigr).
      \quad
    \end{multline*}
  \end{enumroman}
\end{lemma}

\begin{proof}
  Applying the diagonal to the \((\n+1)\)-simplex~\(z=h_{\alpha,\beta}(x,y)\) gives
  \begin{equation}
    \label{eq:p0}
    \Delta z = \sum_{l=0}^{\n+1} \restr{z}{0}{l}\otimes \restr{z}{l}{\n+1}.
  \end{equation}
  For each of the three claims we will look at each individual summand corresponding to some value~\(0\le l\le\n+1\).

  \ref{thm:GkGl-1}:
  We consider the term
  \begin{multline}
    \label{eq:p1}
    \qquad
    \pm\bigl((p_{X}\otimes p_{Y})^{\otimes k}\,\Psi_{\jj}(\restr{z}{0}{l})\bigl) \otimes \restr{z}{l}{\n+1} \\
    \in \bigl(C(X)\otimes C(Y)\bigr)^{\otimes k}\otimes C(X\times Y),
    \qquad
  \end{multline}
  appearing in~\((\GG_{k}\otimes 1)\,\Delta\,z\), where \(\jj\) is a decomposition of~\(k-1\).
  From the identity~\(\Psi_{\jj}=\AWu{u(\jj)}\) and the specific form of~\(u(\jj)\) given in~\eqref{eq:u-a-b}
  we know that all components in the \(k\)~factors lying in~\(C(Y)\) have degree at least~\(1\).

  Let \([i_{1},j_{1}]\) be the first interval of~\(\beta\).
  Then \(p_{Y}(z)\) is degenerate on~\([0,i_{1}]\).
  For~\(l\le i_{1}\) all components lying in the factors~\(C(Y)\) are therefore degenerate
  in~\eqref{eq:p1}, forcing \(G_{k}(\restr{z}{0}{l})=0\).

  \def\nn{l}
  In the case~\(l>i_{1}\), let \([\nn_{t},\nn_{t+1}]\) be the interval for the component in the first factor~\(C(Y)\)
  for a fixed interval cut
  \begin{equation}
    0 = \nn_{0} \le \nn_{1} \le \dots \le \nn_{3k-1} = l
  \end{equation}
  of the interval~\([0,l]\). As mentioned above, we can assume \(\nn_{t}<\nn_{t+1}\).
  In fact, we must have \(i_{1}=\nn_{t}<\nn_{t+1}\le j_{1}+1\) since \(p_{X}(z)\) is degenerate on~\([i_{1},j_{1}+1]\).
  But this implies that in~\eqref{eq:p1} the component in some factor~\(C(X)\) is
  of the form~\(p_{X}(z(J))\) with~\(\nn_{t}\),~\(\nn_{t+1}\in J\).
  Since \(p_{X}(z)\) is degenerate on~\([\nn_{t},\nn_{t+1}]\), this restriction~\(p_{X}(z(J))\) is degenerate.

  \ref{thm:GkGl-2}:
  By assumption, \(\beta\) has at least two intervals~\([i_{1},j_{1}]\) and~\([i_{2},j_{2}]\).

  If \(l>j_{1}+1\), then \(p_{X}(z)\) is degenerate on ~\([i_{1},j_{1}+1]\)
  while \(p_{Y}(z)\) is degenerate on~\([j_{1}+1,i_{2}]\). This implies that there is no way
  to divide \(\restr{z}{0}{l}\) into a front face~\(\restr{z}{0}{l'}\) and a back face~\(\restr{z}{l'}{l}\) such that both
  \(p_{X}(\restr{z}{0}{l'})\) and~\(p_{Y}(\restr{z}{l'}{l})\)
  are non-degenerate. Hence \(\AW(\restr{z}{0}{l})=0\).

  So assume \(l\le j_{1}+1\). Then \(p_{X}(z)\) is degenerate on~\([i_{2},j_{2}+1]\),
  and \(p_{Y}(z)\) is degenerate on~\([p+q+1,p+q+2]\) where \(j_{2}+1\le p+q+1\).
  Hence if \(n=2\), we get
  \begin{equation}
    G_{n-1}(\restr{z}{l}{\n+1}) = \AW(\restr{z}{l}{\n+1}) = 0
  \end{equation}
  by reasoning as before. If \(n\ge3\), then for
  \begin{equation}
    \label{eq:p2}
    G_{n-1}(\restr{z}{l}{\n+1}) \in \bigl(C(X)\otimes C(Y)\bigr)^{\otimes(n-1)}
  \end{equation}
  to be non-zero, some component in one of the tensor factors~\(C(Y)\)
  must be of the form~\(p_{Y}(\restr{z}{i_{2}}{j_{2}+1})\). But then some component in one of the tensor factors~\(C(X)\)
  is of the form~\(p_{X}(z(J))\) with~\(i_{2}\),~\(j_{2}+1\in J\).
  As before, this restricted simplex is degenerate
  because \(p_{X}(z)\) is degenerate on the interval~\([i_{2},j_{2}+1]\).

  \ref{thm:GkGl-3}:
  By assumption, \(p_{X}(z)\) is degenerate on~\([\si,\si+q+1]\), and
  so is \(p_{Y}(z)\) on \([\si+q+1,p+q+2]\). 
  Hence in~\eqref{eq:p0} we must have \(l\le \si+q+1\).
  Assume \(l<\si+q+1\) and set \(t=\max(l,\si)\). Then \(p_{X}(z)\) is degenerate on~\([t,\si+q+1]\) while
  \(p_{Y}(z)\) is degenerate on~\([l,t]\) 
  and on~\([\si+q+1,p+q+2]\). Hence for~\eqref{eq:p2}
  to be non-zero, the component in the first tensor factor~\(C(Y)\) has to be \(p_{Y}(\restr{z}{t}{\si+q+1})\).
  But then the component in some tensor factor~\(C(X)\) must be
  of the form~\(p_{X}(z(J))\) with~\(t\),~\(\si+q+1\in J\), hence degenerate.

  This shows that \(l=\si+q+1\) is the only value with a non-zero contribution.
  The claim now follows from the formula~\eqref{eq:def-psi} for~\(z=h_{\alpha,\beta}(x,y)\) and the identity
  \begin{equation}
    \AW(\restr{z}{0}{l}) = p_{X}(\restr{z}{0}{\si}) \otimes p_{Y}(\restr{z}{\si}{\si+q+1}).
  \end{equation}
  by applying the appropriate face operators.
\end{proof}

We finally show \eqref{eq:GGn-Gn}
by induction on~\(n\). The identity holds for~\(n=1\) since both
\begin{equation}
  \GG_{1} = (p_{X}\otimes p_{Y})\,\Psi_{1} = (p_{X}\otimes p_{Y})\,(E_{0}\otimes 1)\,\Delta^{(2)} = (p_{X}\otimes p_{Y})\,\Delta
\end{equation}
and~\(G_{1}\) are equal to the Alexander--Whitney map~\(\AW\).

Now assume the claim to hold for all~\(1\le k<n\), and consider \((x,y)\in X_{\n}\times Y_{\n}\) for some~\(m\ge0\).
By the definition~\eqref{eq:GM-rec-def} of~\(G_{n}\), induction and \Cref{thm:GkGl} we have
\begin{align}
  \MoveEqLeft{G_{n}(x,y) = -\sum_{l=1}^{n-1} (-1)^{n-l}\,(G_{l}\otimes G_{n-l})\,\Delta\,h(x,y)} \\
  \notag &= -\sum_{l=1}^{n-1} (-1)^{n-l}\,(\GG_{l}\otimes\GG_{n-l})\,\Delta\,h(x,y) \\
  \notag &= \!\! \sum_{\substack{0\le p+q < \n\\(\alpha,\beta)\vdash(p,q+1)}} \!\!\!\! (-1)^{p+q+(\alpha,\beta)}\,
    (\AW\otimes\GG_{n-1})\,\Delta\,h_{\alpha,\beta}\bigl(\partial_{p+1}^{p+q}\,x,\partial_{0}^{p-1}\,y\bigr),
\end{align}
where the last sum only extends over partitions~\((\alpha,\beta)\) where \(\beta=(\si,\dots,\si+q)\) is an interval for some~\(0\le \si\le p\).

Let us look at an individual term
\begin{multline}
  \label{eq:g-G}
  \qquad
  (-1)^{n+p+q+(\alpha,\beta)}\,
  (\AW\otimes\GG_{n-1})\,\Delta\, h_{\alpha,\beta}\bigl(\partial_{p+1}^{p+q}\,x,\partial_{0}^{p-1}\,y\bigr) \\
  = (-1)^{\epsilon}\,\bigl(\restr{x}{0}{\si} \otimes \restr{y}{p}{p+q+1}\bigr) \otimes \GG_{n-1}(\tilde{x},\tilde{y}).
  \qquad
\end{multline}
Here
\begin{align}
  \tilde{x} &= x(\si\dots p,p+q+1\dots \n), \\
  \tilde{y} &= s_{[0,p-\si]}\,y(p+q+1\dots \n) \\
  \shortintertext{are simplices of dimension~\(\n'=\n-\si-q+1\), and}
  \label{eq:delta-total}
  \epsilon &= p+q+(p-\si)(q+1) + n(\si+q).
\end{align}
Consider a fixed decomposition~\(\jj'=(j'_{1},\dots,j'_{n-1})\) of~\(n-2\) appearing in
\begin{equation}
  \label{eq:GG-n1}
  \GG_{n-1}(\tilde{x},\tilde{y}) = (-1)^{n-2}\,(p_{X}\otimes p_{Y})^{\otimes(n-1)}\sum_{\jj'}\Psi_{\jj'}(\tilde{x},\tilde{y})
\end{equation}
and a fixed cut
\def\nn{m}
\begin{equation}
  0=\nn'_{0}\le \nn'_{1}\le \dots\le \nn'_{3(n-1)-1} = m'
\end{equation}
of~\([0,\n-\si-q+1]\) into \(2(n-1)\)~intervals. These choices lead to a tensor product of simplices
\begin{equation}
  \bigl(\sweehgc{a}{1} \otimes \sweehgc{b}{1}\bigr) \otimes \dots \otimes \bigl(\sweehgc{a}{n-1} \otimes \sweehgc{b}{n-1}\bigr) \in (C(X)\otimes C(Y))^{\otimes(n-1)}.
\end{equation}
The effect of the degeneracy operator~\(s_{[0,p-\si]}\) in the formula for~\(\tilde{y}\) is
that in order to obtain a non-degenerate simplex~\(\sweehgc{b}{1}\), the simplex~\(\tilde{y}\) cannot be restricted to an interval
that overlaps with the interval~\([0,p-\si+1]\). Hence this latter interval is completely covered
by intervals corresponding to terms~\(\sweehgc{a}{t}\) or, in other words, intervals with odd labels~\(v_{t}\)
from the surjection~\(u(\jj')\) given by~\eqref{eq:u-a-b}.

Suppose that the interval containing \([p-\si,p-\si+1]\) has label~\(v_{t}=2t-1\) in~\(u(\jj')\).
This implies \(j'_{1}=\dots=j'_{t-1}=0\) because the corresponding odd labels~\(<2t-1\) cannot enclose any even label. Moreover,
\begin{equation}
  \label{eq:a-b-jj}
  \bigl(\restr{x}{0}{\si} \otimes \restr{y}{p}{p+q+1}\bigr) \otimes
  \bigl(\sweehgc{a}{1} \otimes \sweehgc{b}{1}\bigr) \otimes \dots \otimes \bigl(\sweehgc{a}{n-1} \otimes \sweehgc{b}{n-1}\bigr)
\end{equation}
is a term appearing in~\((p_{X}\otimes p_{Y})^{\otimes n}\,\Psi_{\jj}(x,y)\), hence in~\(\GG_{n}(x,y)\), for
\begin{equation}
  \jj = (\underbrace{0,\dots,0}_{\text{\(t\)~times}},j'_{t}+1,j'_{t+1},\dots,j'_{n-1}),
\end{equation}
namely for the interval cut
\begin{multline}
  0=\nn_{0}\le \nn_{1}=\si\le \nn_{2} = \nn'_{1}+\si \le \dots \le \nn_{t} = \nn'_{t-1}+\si \le \nn_{t+1} = p \\
  \le \nn_{t+2} = p+q+1 \le \nn_{t+3} = \nn'_{t} \le \dots \le \nn_{3n-1} = \nn'_{3(n-1)-1} = m'.
\end{multline}
Conversely, by reversing this process we see that for fixed~\(p\),~\(q\) and~\(\si\) all terms~\eqref{eq:a-b-jj} arise this way
from some~\(\jj'\) and some interval cut of~\([0,m']\).

Let us compare the signs. In addition to~\eqref{eq:GG-n1} we have
\begin{equation}
  \label{eq:comp-Psi}
  \GG_{n} = (-1)^{n-1}\,(p_{X}\otimes p_{Y})^{\otimes n}\sum_{\jj} \Psi_{\jj}.
\end{equation}
We again focus on fixed values~\(p\),~\(q\) and~\(\si\) and a fixed decomposition~\(\jj'\) of~\(n-2\) together with a fixed interval cut,
giving an interval cut for a decomposition~\(\jj\) of~\(n-1\) as before.

The simplex~\(\sweehgc{b}{1}\), which is the first component in~\(C(Y)\), does not show up in~\(\jj'\). In~\(\jj\), it produces a permutation sign
since the corresponding (final) interval has to be moved past all intervals between vertices~\(\si\) and~\(p\).
These intervals are all final except for the last one since the two intervals enclosing the one labelled~\(w_{1}=2\) have the same (odd) label.
Hence the change in the permutation sign exponent is
\begin{equation}
  \label{eq:delta-perm}
  \Delta\epsilon_{\mathrm{perm}} = (p-\si+1)(q+1) = (p-\si)(q+1)+q+1.
\end{equation}
As far as changes in the positional signs are concerned, there are two kinds of contributions.
Recall that only inner intervals contribute to the positional sign.
There are \(n-1\)~inner intervals in~\(\jj\) and \(n-2\) in~\(\jj'\).
The first inner interval in~\(\jj\) comes right before the interval with labelled~\(w_{1}=2\) (hence ending at vertex~\(p\)),
and it disappears when passing to~\(\jj'\). The others bijectively correspond to the inner intervals in~\(\jj'\),
but they are shifted \(\si+q\)~positions to the left. The reason is that the vertex~\(p+q+1\) in~\((x,y)\)
corresponds to the vertex~\(p-\si+1\) in~\((\tilde{x},\tilde{y})\), and these inner intervals appear past these vertices.
Hence the change in the positional sign exponent is
\begin{equation}
  \label{eq:delta-pos}
  \Delta\epsilon_{\mathrm{pos}} = p + (\si+q)(n-2) \equiv p + n(\si+q) \pmod{2}.
\end{equation}
Adding up the sign changes (including the one between~\eqref{eq:GG-n1} and~\eqref{eq:comp-Psi}), we get
\begin{equation}
  \epsilon = \Delta\epsilon_{\mathrm{perm}} + \Delta\epsilon_{\mathrm{pos}} - 1
\end{equation}
Hence
\begin{equation}
  G_{n}(x,y) = (-1)^{n}\,(\AW \otimes\GG_{n-1})\,\Delta\,h(x,y) = \GG_{n}(x,y),
\end{equation}
which completes the proof of \Cref{thm:Phihga-PhiGM}.

\section{Comparison of the two products on~\texorpdfstring{\(\Tor\)}{Tor}}
\label{sec:compare-prod}

As discussed in \Cref{rem:product-E2}, the \(\Ainf\)-structure on the two-sided bar construction
induces, under certain assumptions,
the canonical product on the second page of the associated Eilenberg--Moore spectral sequence.
The question remains whether the new product on the differential torsion product agrees with
the ``original'' one defined previously by Smith~\cite{Smith:1967}
by dualizing the construction of Eilenberg--Moore~\cite[Sec.~18]{EilenbergMoore:1966}.
This is necessarily the case in the situation of the Eilenberg--Moore \Cref{thm:em-ainf} since
for both products we have a multiplicative isomorphism with~\(H^{*}(\gE)\).

Let us assume that \(\kk\) is a principal ideal domain.
Recall from~\cite[Prop.~I.3.4]{Smith:1967} or~\cite[Prop.~7.17]{McCleary:2001}
that the original product is defined via the maps\footnote{%
  The map~\(\Tor_{j}(j,j)\) is missing in~\cite{Smith:1967}.}
\begin{equation}
  \label{eq:def-prod-Tor-orig}
  \begin{tikzcd}[row sep=scriptsize]
    \Tor_{C^{*}(B)}\bigl(C^{*}(X),C^{*}(E)\bigr) \otimes \Tor_{C^{*}(B)}\bigl(C^{*}(X),C^{*}(E)\bigr) \arrow{d}{\shuffle} \\
    \Tor_{C^{*}(B)\otimes C^{*}(B)}\bigl(C^{*}(X)\otimes C^{*}(X),C^{*}(E)\otimes C^{*}(E)\bigr)
      \arrow{d}{\Tor_{j}(j,j)} \\
    \Tor_{(C(B)\otimes C(B))^{\vee}}\bigl((C(X)\otimes C(X))^{\vee},(C(E)\otimes C(E))^{\vee}\bigr) \\
    \Tor_{C^{*}(B\times B)}\bigl(C^{*}(X\times X),C^{*}(E\times E)\bigr)
      \arrow{u}[right]{\Tor_{\shuffle^{*}}(\shuffle^{*},\shuffle^{*})}[left]{\cong} \arrow{d}{\Tor_{\Delta^{*}}(\Delta^{*},\Delta^{*})} \\
    \Tor_{C^{*}(B)}\bigl(C^{*}(X),C^{*}(E)\bigr) \mathrlap{.}
  \end{tikzcd}
\end{equation}
The top arrow is the shuffle map~\eqref{eq:def-shuffle-twosided} in cohomology.
The map~\(\shuffle^{*}\) appearing in the third arrow is induced by the transpose of the shuffle map~\eqref{eq:ez-contraction}
(with~\(X=Y\), and analogously for~\(B\) and~\(E\)).
This shuffle map is a quasi-isomorphism of dgcs and \(\shuffle^{*}\) therefore one of dgas,
which implies that \(\Tor_{\shuffle^{*}}(\shuffle^{*},\shuffle^{*})\) is an isomorphism, see~\cite[Cor.~1.8, proof of Cor.~3.5]{GugenheimMay:1974}.
Moreover, \(j\) denotes the canonical chain map
\begin{equation}
  C^{*}(X)\otimes C^{*}(X) \to \bigl(C(X)\otimes C(X)\bigr)^{\vee}.
\end{equation}

\begin{theorem}
  \label{thm:comparison-products}
  Let \(B\),~\(X\) and~\(E\) be simplicial sets, and let \(\kk\) be a principal ideal domain.
  The product~\eqref{eq:def-prod-Tor-orig} on~\(\Tor_{C^{*}(B)}\bigl(C^{*}(X),C^{*}(E)\)
  agrees with the one induced by the \(\Ainf\)-structure on \(\BB(C^{*}(X),C^{*}(B),C^{*}(E)\)
  defined in \Cref{sec:ainf-twosided} or, in other words, by the Carlson--Franz product~\(m_{2}\).
\end{theorem}

As mentioned in the introduction, this result has independently been obtained by Carlson~\cite{Carlson:product}.

\begin{proof}
  We assume \(\kk\) to be a principal ideal domain only to ensure that the cohomology
  of the two-sided bar construction is indeed the stated torsion product.

  As before, \(G\) denotes the Gugenheim--Munkholm shc map for the Eilenberg--Zilber contraction.
  From the transpose
  \begin{equation}
    \shuffle^{*} \circ G^{*} = 1 
  \end{equation}
  of the identity~\eqref{eq:f-G-id} (with~\(f=\shuffle\)) we deduce
  \begin{equation}
    \Tor_{\shuffle^{*}}(\shuffle^{*},\shuffle^{*})^{-1} = \Tor_{G^{*}}(G^{*},G^{*}). 
  \end{equation}
  Together with the definition of~\(m_{2}\) recalled in \Cref{rem:m2-hga},
  the definition~\eqref{eq:def-PhiGM} of~\(\PhiGM\) and \Cref{thm:Phihga-PhiGM} 
  we obtain for the Carlson--Franz multiplication on the differential torsion product
  \begin{align}
    H^{*}(m_{2}) &= \Tor_{\Phihga}(\Phihga,\Phihga)\,\shuffle
    = \Tor_{\PhiGM}(\PhiGM,\PhiGM)\,\shuffle \\
    \notag &= \Tor_{\Delta^{*}}(\Delta^{*},\Delta^{*})\,\Tor_{G^{*}}(G^{*},G^{*})\,\Tor_{j}(j,j)\,\shuffle \\
    \notag &= \Tor_{\Delta^{*}}(\Delta^{*},\Delta^{*})\,\Tor_{\shuffle^{*}}(\shuffle^{*},\shuffle^{*})^{-1}\Tor_{j}(j,j)\,\shuffle,
  \end{align}
  which is the composition displayed in~\eqref{eq:def-prod-Tor-orig}.
\end{proof}

\end{document}